\documentclass[12pt,oneside,english]{amsart}

\usepackage[T1]{fontenc}
\usepackage[latin9]{inputenc}
\usepackage{geometry}
\usepackage{amstext}
\usepackage{amsthm}
\usepackage{amssymb}

\makeatletter
\numberwithin{equation}{section}
\numberwithin{figure}{section}
\theoremstyle{plain}
\newtheorem{thm}{\protect\theoremname}
\theoremstyle{plain}
\newtheorem{lem}[thm]{\protect\lemmaname}
\theoremstyle{plain}
\newtheorem{prop}[thm]{\protect\propositionname}

\usepackage[numbers,sort&compress,square]{natbib}
\usepackage{url}
\usepackage{hyperref}

\makeatother

\usepackage{babel}
\providecommand{\lemmaname}{Lemma}
\providecommand{\propositionname}{Proposition}
\providecommand{\theoremname}{Theorem}

\begin{document}
\title{The TAP-Plefka variational principle for the spherical SK model}
\author{David Belius, Nicola Kistler}
\begin{abstract}
We reinterpret the Thouless-Anderson-Palmer approach to mean field
spin glass models as a variational principle in the spirit of the
Gibbs variational principle and the Bragg-Williams approximation.
We prove this TAP-Plefka variational principle rigorously in the case
of the spherical Sherrington-Kirkpatrick model.
\end{abstract}

\maketitle

\section{Introduction}

There are several approaches in theoretical physics and mathematics
to study the Sherrington-Kirkpatrick (SK) mean field spin glass model
\cite{SKSolvableModelOfASpinGlass} and its variants. The most successful
in physics is the replica approach, which with Parisi's replica symmetry
breaking Ansatz led him to his celebrated formula for the free energy
\cite{MezardParisiVirasoro-SpinGlassTheoryandBeyond}. The mathematically
rigorous proofs of the formula due to Guerra, Talagrand and Panchenko
are based on a subtle combination of interpolation, recursion, the
Ghirlanda-Guerra identities and an invariance property for the limiting
Gibbs measure \cite{GuerraBrokenReplicaSymmetryBounds,TalagrandTheParisiFormula,PanchenkoTheSKModel,PanchenkoTheParisiUltrametricityConjecture}.
A further approach in the physics literature is the one due to Thouless,
Anderson and Palmer (TAP) and Plefka. It originates in \cite{TAPSolutionOfSolvableModelOfASpinGlass}
as a diagrammatic expansion of the partition function of the Ising
SK model relating the free energy to the so called TAP free energy,
which is a disorder-dependent function defined on the space of magnetizations
of the spins. It claims that the free energy  equals the TAP free
energy at magnetizations that solve a set of mean field equations
and satisfy certain convergence conditions, which have not been completely
clarified. \emph{Plefka's condition} \cite{PlefkaALBfortheSpinGlassOrderParameter,PlefkaConvergenceCondOftheTAPequations}
is believed to be necessary, but it is not clear if it is also sufficient.
The high temperature analysis of \cite{TAPSolutionOfSolvableModelOfASpinGlass}
has been made rigorous in \cite{AizenmanLebowitzRuelleSomeRigorousResultsOnTheSherringtonKirckpatrick}.
The physicist's TAP approach has been adapted to spherical models
in \cite{CrisantiSommersTAPApproachtoSphericalPspinSGModels}.

In this paper we reinterpret the TAP approach as a variational principle
for the free energy, which states that the free energy equals the
maximum of the TAP free energy taken over magnetizations satisfying
appropriate conditions. We make this rigorous in the case of the spherical
Sherrington-Kirkpatrick model, and show that for this model Plefka's
condition is the only condition needed to formulate the variational
principle.

Let $H_{N}\left(\sigma\right),\sigma\in\mathbb{R}^{N}$, be the $2$-spin
spherical SK Hamiltonian which is a centered Gaussian process on $\mathbb{R}^{N}$
with covariance 
\begin{equation}
\mathbb{E}\left[H_{N}\left(\sigma\right)H_{N}\left(\sigma'\right)\right]=N\left(\sigma\cdot\sigma'\right)^{2},\label{eq: covar}
\end{equation}
which can be constructed by setting
\begin{equation}
H_{N}\left(\sigma\right)=\sqrt{N}\sum_{i,j=1}^{N}J_{ij}\sigma_{i}\sigma_{j}\label{eq: explicit construction}
\end{equation}
for iid standard Gaussian random variables $J_{ij}$ and $\sigma\in\mathbb{R}^{N}$.
Let $E$ be the uniform measure on the unit sphere in $\mathbb{R}^{N}$
and let
\begin{equation}
Z_{N}\left(\beta,h_{N}\right)=E\left[e^{\beta H_{N}\left(\sigma\right)+Nh_{N}\cdot\sigma}\right]\text{ and }F_{N}\left(\beta,h_{N}\right)=\frac{1}{N}\log Z_{N}\left(\beta,h_{N}\right)\label{eq: part func and FE}
\end{equation}
be the partition function and free energy in the presence of an external
field $h_{N}\in\mathbb{R}^{N}$. The TAP free energy for this model
is given by \cite{TAPSolutionOfSolvableModelOfASpinGlass,CrisantiSommersTAPApproachtoSphericalPspinSGModels}
\[
H_{TAP}\left(m\right)=\beta H_{N}\left(m\right)+Nm\cdot h_{N}+\frac{N}{2}\log\left(1-\left|m\right|^{2}\right)+N\frac{\beta^{2}}{2}\left(1-\left|m\right|^{2}\right)^{2}
\]
for $m\in\mathbb{R}^{N}$ with $\left|m\right|<1$, and Plefka's condition
\cite{PlefkaConvergenceCondOftheTAPequations,PlefkaALBfortheSpinGlassOrderParameter}
reads
\[
\beta\left(m\right)\le\frac{1}{\sqrt{2}},
\]
 where
\[
\beta\left(m\right)=\beta\left(1-\left|m\right|^{2}\right).
\]
We refer to the approximation
\begin{equation}
F_{N}\left(\beta,h_{N}\right)\approx\frac{1}{N}\sup_{m\in\mathbb{R}^{N}:\left|m\right|<1,\beta\left(m\right)\le\frac{1}{\sqrt{2}}}H_{TAP}\left(m\right)\label{eq: TAP-Plefka var princip}
\end{equation}
as the TAP-Plefka variational principle and prove it in the following
form.
\begin{thm}
\label{thm: main intro II}For any $\beta>0$, $h\ge0$ and any sequence
$h_{1},h_{2},\ldots$ with $\left|h_{N}\right|=h$ one has 
\begin{equation}
\left|F_{N}\left(\beta,h_{N}\right)-\frac{1}{N}\sup_{m\in\mathbb{R}^{N}:\left|m\right|<1,\beta\left(m\right)\le\frac{1}{\sqrt{2}}}H_{TAP}\left(m\right)\right|\to0\text{ in probability}.\label{eq: main intro II}
\end{equation}
\end{thm}

We also include a solution of the TAP-Plefka variational problem that
reduces it from a random $N$-dimensional optimization problem to
one which is deterministic and one dimensional.
\begin{lem}
\label{lem: TAP var sol}For any $\beta,h,h_{1},h_{2},\ldots$, as
in Theorem \ref{thm: main intro II} one has 
\[
\left|\frac{1}{N}\sup_{m\in\mathbb{R}^{N}:\left|m\right|<1,\beta\left(m\right)\le\frac{1}{\sqrt{2}}}H_{TAP}\left(m\right)-{\displaystyle {\displaystyle \sup_{q\in\left[0,1\right]:\beta\left(1-q\right)\le\frac{1}{\sqrt{2}}}}\mathcal{B}\left(q\right)}\right|\to0\text{ in probability},
\]
where
\[
\mathcal{B}\left(q\right)=\mathcal{B}\left(q;\beta,h\right)=\sqrt{h^{2}q+2\beta^{2}q^{2}}+\frac{1}{2}\log\left(1-q\right)+\frac{\beta^{2}}{2}\left(1-q\right)^{2}.
\]
\end{lem}

Together, Theorem \ref{thm: main intro II} and Lemma \ref{lem: TAP var sol}
show that 
\begin{equation}
F_{N}\left(\beta,h_{N}\right)\to\sup_{q\in\left[0,1\right]:\beta\left(1-q\right)\le\frac{1}{\sqrt{2}}}\mathcal{B}\left(q\right).\label{eq: free energy M}
\end{equation}
For comparison, the Parisi formula in this context \cite{TalagrandFEOftheSphericalMeanFieldModel,CrisantiSommersTheSphericalPspinInteractionSGModel}
states that
\begin{equation}
F_{N}\left(\beta,h_{N}\right)\to\inf_{q\in\left[0,1\right]}\mathcal{P}\left(q\right),\label{eq: parisi}
\end{equation}
where
\[
\mathcal{P}\left(q\right)=\frac{1}{2}h^{2}\left(1-q\right)+\frac{1}{2}\frac{q}{1-q}+\frac{1}{2}\log\left(1-q\right)+\frac{1}{2}\beta^{2}\left(1-q^{2}\right).
\]

\subsection{Discussion}

\subsubsection{The TAP-Plefka variational principle}

The TAP-Plefka variational principle (\ref{eq: TAP-Plefka var princip})
should be compared to the classical Gibbs variational principle which
states that
\begin{equation}
F_{N}\left(\beta,h_{N}\right)=\frac{1}{N}\sup_{\mathcal{G}}\left\{ \mathcal{G}\left(\beta H_{N}\left(\sigma\right)+N\sigma\cdot h_{N}\right)-H\left(\mathcal{G}||E\right)\right\} ,\label{eq: Gibbs}
\end{equation}
where the supremum is over all probability measures which are absolutely
continuous with respect to $E$, and $H\left(\mathcal{G}||E\right)$
is the relative entropy of $\mathcal{G}$ with respect to $E$. The
first term is the internal energy and the second is the entropy.

In the classical Bragg-Williams approximation \cite[Section 4.1.2]{VilfanLectureNotesinStatisticalMechanics,BraggWilliamsTheEffectofThermalAgitationonAtomicArrangementInAlloys}
in non-disordered statistical physics one restricts this $\sup$ to
simple measures $\mathcal{G}$ that are parameterized by a mean magnetization
$m\in\mathbb{R}^{N}$; in the case of $\pm1$ spins one considers
measures under which the spins $\sigma_{i}$ are independent with
mean $m_{i}$. For any $m$ the corresponding measure gives a lower
bound for the free energy, because of the Gibbs variational principle.
If the Bragg-Williams approximation is successful, maximizing over
$m$ yields the true free energy (at least to leading order). If applied
to approximate the free energy of the Curie-Weiss Hamiltonian $\frac{\beta}{N}\sum_{i,j}\sigma_{i}\sigma_{j}+h\sum_{i=1}^{N}\sigma_{i}$
one obtains a variational problem over $m\in\mathbb{R}^{N}$ that
is equivalent to
\[
\frac{1}{N}\sup_{\bar{m}\in\left[-1,1\right]}\left\{ \beta\bar{m}^{2}+h\bar{m}-\frac{1+\bar{m}}{2}\log\left(\frac{1-\bar{m}}{2}\right)-\frac{1-\bar{m}}{2}\log\left(\frac{1+\bar{m}}{2}\right)\right\} ,
\]
which also appears in the classical solution of the model via the
large deviation rate function of the binomial distribution, and is
thus indeed an accurate approximation.

In the spherical setting a product measure on the spins is not absolutely
continuous with respect to $E$, but a natural family of measures
is provided by exponential tilts of the uniform distribution given
by $e^{\lambda\sigma\cdot m}dE$ appropriately normalized, for $\lambda=\lambda\left(m\right)$
chosen so that the mean magnetization is $m$. For such a measure
the internal energy will be close to $\beta H_{N}\left(m\right)+Nm\cdot h_{N}$
and the entropy will be close to $\frac{N}{2}\log\left(1-\left|m\right|^{2}\right)$.
Thus the Bragg-Williams approximation of the free energy is
\[
\frac{1}{N}\sup_{m\in\mathbb{R}^{N}:\left|m\right|<1}\left\{ \beta H_{N}\left(m\right)+Nm\cdot h_{N}+\frac{N}{2}\log\left(1-\left|m\right|^{2}\right)\right\} ,
\]
which is in fact inaccurate, in light of (\ref{eq: TAP-Plefka var princip}).
However, the TAP-Plefka variational principle can be seen as the appropriate
modification of the Bragg-Williams approximation to obtain an accurate
approximation for this disordered system, by adding the Onsager correction
term $\frac{N}{2}\beta^{2}\left(1-\left|m\right|^{2}\right)^{2}$
and restricting the $\sup$ to $m$-s satisfying Plefka's condition.

\subsubsection{The $2$-spin model}

The $2$-spin spherical SK model, which is the model we consider
in this paper, is a much simpler model than the other Ising and spherical
SK variants. It is always replica symmetric, for all inverse temperatures
$\beta$ and external field strengths $h$, and the Parisi formula
can be written as a one parameter variational principle (see (\ref{eq: parisi})).
If the external field vanishes ($h=0$) an explicit closed form (non-variational)
formula for $\lim_{N\to\infty}F_{N}$ exists, even in low temperature.

Furthermore, the Hamiltonian can be written as $H_{N}\left(\sigma\right)=\sqrt{N}\sigma^{T}S_{N}\sigma$
for a random matrix $S_{N}$ from the Gaussian orthogonal ensemble,
and by the rotational invariance of the sphere we can work in the
diagonalizing basis of $S_{N}$, in which case $H_{N}\left(\sigma\right)=\sqrt{N}\sum_{i=1}^{N}\lambda_{i}\sigma_{i}^{2}$
where $\lambda_{i}$ are the eigenvalues of $S_{N}$. Because of this
the free energy can be computed by a random matrix approach, without
using the Parisi formula \cite{KosterlitzThoulessJonesSphericalModelofASpinGlass,BaikLeeFluctuationsOfTheFEofTheSphericalSK,GeneoveseTantariLegendrDualityOfSphericalandGaussianSpinGlasses}.
Part of our analysis also relies on random matrix considerations.
The resulting formulas (\ref{eq: main intro II}) and (\ref{eq: free energy M})
are not related to previously obtained formulas for the free energy.
Our proof is the first rigorous derivation of a TAP variational principle
based on a microcanonical analysis that yields bounds valid for finite
$N$, and where Plefka's condition appears naturally. 

\subsubsection{Previous work in the mathematical literature}

Recently in \cite{ChenPanchekoOntheTAPFEInTheMixedPspinModels} Chen
and Panchenko used the Parisi formula to verify a TAP variational
principle for mixed Ising SK models in the thermodynamic limit, that
is an equality after taking the limit $N\to\infty$, with a different
condition replacing Plefka's condition. In \cite{SubagGeometryOfGibbsMeasure}
Subag constructs for very low temperatures a decomposition of the
Gibbs measure of pure $p$-spin spherical models into pure states
in a microcanonical fashion, and notices that the log of the weight
of each pure state coincides with its TAP free energy.

Further mathematical results concern the TAP equations. These are
a system of nonlinear equations for the quenched mean magnetization
which have been interpreted as a self-consistency property; within
our framework it is natural to view the TAP equation as the critical
point equations of the TAP free energy. Bolthausen has developed an
iterative scheme for solving the TAP equations for the Ising SK model
\cite{BolthausenAnIterativeConstructionOfSoloftheTAPequations} that
converges in the whole conjectured high temperature regime. Talagrand
\cite{talagrand2003spin} and Chatterjee \cite{ChatterjeeSpinGlassesandSteinsMethod}
showed that in high enough temperature the mean magnetization of the
Ising SK Gibbs measure satisfies the TAP equations. Auffinger and
Jagganath have used the Parisi formula to prove that solutions of
the TAP equations describe the magnetization inside appropriately
defined pure states of generic mixed Ising models for all temperatures
\cite{AuffingerJagannathTAPequationsforcondGibbsmeasuresingeneric}.
Auffinger, Ben Arous \& Cerny have studied the (annealed) complexity
of TAP solutions for pure $p$-spin spherical Hamiltonians \cite{AuffingerBenArousCernyComplexityofSpinGlasses}.

\subsection{A word on the proof}

The proof of Theorem \ref{thm: main intro II} splits into a proof
of a lower bound and a proof of an upper bound for the partition function
$Z_{N}\left(\beta,h_{N}\right)$. Both are based on recentering the
Hamiltonian around magnetizations $m$ of potential pure states (a
similar recentering has been used by TAP \cite{TAPSolutionOfSolvableModelOfASpinGlass},
Bolthausen \cite{BolthausenPrivateCommunication} and Subag \cite{SubagGeometryOfGibbsMeasure}).
In general, recentering around a given $m$ gives rise to an effective
external field for the recentered Hamiltonian.

The lower bound is presented in Section \ref{sec: LB} and is proved
by considering a recentering around any magnetization $m$ that satisfies
Plefka's condition. We then restrict the integral in $Z_{N}\left(\beta,h_{N}\right)$
to a subset of the sphere which is ``centered at $m$'', namely
the intersection of the sphere with a plane that contains $m$ and
is perpendicular to both $m$ and the effective external field. The
mean energy (value of Hamiltonian and external field) on this subset
is $\beta H_{N}\left(m\right)+Nm\cdot h_{N}$, cf. the first two terms
of $H_{TAP}\left(m\right)$. The log of the measure of the subset
is approximately $\frac{N}{2}\log\left(1-\left|m\right|^{2}\right)$,
cf. the third term of $H_{TAP}\left(m\right)$. Finally the recentered
Hamiltonian on this subset turns out to be a $2$-spin Hamiltonian
on a lower dimensional sphere without external field at inverse temperature
$\beta\left(m\right)=\beta\left(1-\left|m\right|^{2}\right)$. If
Plefka's condition is satisfied this is less than the critical inverse
temperature $\beta_{c}=\frac{1}{\sqrt{2}}$, and it is therefore natural
that using the uniform measure on the subset as a reference measure
the free energy of the recentered Hamiltonian is $\frac{1}{2}\beta\left(m\right)^{2}=\frac{1}{2}\beta^{2}\left(1-\left|m\right|\right)^{2}$,
cf. the last term of $H_{TAP}\left(m\right)$ (the Onsager correction).
In this way we show that the subset contributes approximately $\exp\left(H_{TAP}\left(m\right)\right)$
to $Z_{N}\left(\beta,h\right)$. This shows that $H_{TAP}\left(m\right)$
is a lower bound of the free energy for any $m$ satisfying Plefka's
condition. Note that it also provides a natural interpretation of
the terms in $H_{TAP}\left(m\right)$, and of Plefka's condition as
the condition that a pure state should effectively be in high temperature.

The upper bound is significantly harder and is proved in Section
\ref{sec: UB}. It involves the construction of a low-dimensional
subspace of magnetizations $\mathcal{M}_{N}$ with the property that
after recentering around any $m\in\mathcal{M}_{N}$, the effective
external field is again almost completely contained in $\mathcal{M}_{N}$.
We write the integral in $Z_{N}\left(\beta,h_{N}\right)$ as a double
integral first over $\mathcal{M}_{N}$ and then over the perpendicular
space $\mathcal{M}_{N}^{\bot}$. For a fixed $m\in\mathcal{M}_{N}$
the integral over the perpendicular space $\mathcal{M}_{N}^{\bot}$
is seen to be related to a partition function without external field
at a higher effective temperature, and is shown to be close to the
exponential of a modified TAP energy, with the Onsager correction
$\frac{\beta^{2}}{2}\left(1-\left|m\right|^{2}\right)^{2}$ replaced
by a different, not entirely explicit, expression. The integral in
$Z_{N}\left(\beta,h_{N}\right)$ thus reduces to an integral of the
exponential of the modified TAP energy over the low-dimensional space
$\mathcal{M}_{N}$, and by the Laplace method the log of the integral
turns into the supremum over the modified TAP energy over \emph{all}
$m$. We then show that if the Hessian at a critical point of the
modified TAP energy is negative semi-definite, as it must be at any
local maximum, then $m$ satisfies Plefka's condition and furthermore
the modified TAP energy and the original TAP energy $H_{TAP}\left(m\right)$
are close. From this the upper bound on $Z_{N}\left(\beta,h_{N}\right)$
is seen to follow.

In Section \ref{sec: TAPvarSol} we prove Lemma \ref{lem: TAP var sol}.
In the next section we fix notation and recall some basic facts.

\subsubsection*{Acknowledgments}

The first author thanks Erwin Bolthausen and Giuseppe Genovese for
valuable discussions on a draft of this article. The second author
wishes to express his gratitude to Markus Petermann for a long-standing
discussion on spin glasses, and to Anton Wakolbinger for encouragement.

\section{Notation and basic facts}

The letter $c$ denotes a constant that does not depend on $N$, possibly
a different one each time it is used.

Let $\left(\Omega,\mathcal{A},\mathbb{P}\right)$ be a probability
space with random variables $J_{ij},i,j\ge1$ that are iid standard
Gaussians. Define
\[
H_{N}\left(\sigma\right)=\sqrt{N}\sum_{i,j=1}^{N}J_{i,j}\sigma_{i}\sigma_{j}\text{\,for }\sigma\in\mathbb{R}^{N}.
\]
For any $\lambda\in\mathbb{R}$ and $\sigma\in\mathbb{R}^{N}$ we
have
\begin{equation}
H_{N}\left(\lambda\sigma\right)=\lambda^{2}H_{N}\left(\sigma\right).\label{eq: scaling}
\end{equation}
Let $S_{N}$ be the $N\times N$ matrix given by 
\[
\left(S_{N}\right)_{ij}=\frac{J_{ij}+J_{ji}}{2}.
\]
Note that
\[
H_{N}\left(\sigma\right)=\sqrt{N}\sigma^{T}S_{N}\sigma,
\]
and
\[
\nabla H_{N}\left(\sigma\right)=2\sqrt{N}S_{N}\sigma.
\]
For this reason the $2$-spin Hamiltonian gradient is linear, i.e.
\begin{equation}
\nabla H_{N}\left(\sigma_{1}+\sigma_{2}\right)=\nabla H_{N}\left(\sigma_{1}\right)+\nabla H_{N}\left(\sigma_{2}\right)\text{ for all }\sigma_{1},\sigma_{2}\in\mathbb{R}^{N}.\label{eq: gradient linear}
\end{equation}

We will use, especially in the upper bound, that the empirical spectral
distribution of $S_{N}$ converges to the semi-circle law. Let $\sqrt{N}\theta_{1}^{N}<\ldots<\sqrt{N}\theta_{N}^{N}$
be the eigenvalues of the matrix $S_{N}$. We have that
\[
\frac{1}{N}\sum_{i=1}^{N}\delta_{\theta_{i}^{N}}\to\mu\left(x\right)dx\text{ in distribution, }\mathbb{P}-a.s.,
\]
where 
\begin{equation}
\mu\left(x\right)=\frac{1}{\pi}\sqrt{2-x^{2}}1_{\left[-\sqrt{2},\sqrt{2}\right]}.\label{eq: semi-circle law}
\end{equation}
In addition if we let
\begin{equation}
\theta_{u}=\inf\left\{ \theta:\int_{-\sqrt{2}}^{\theta}\mu\left(x\right)dx=u\right\} ,\label{eq: classical location}
\end{equation}
then
\begin{equation}
\theta_{i}^{N}=\theta_{\frac{i}{N}}+o\left(1\right)\text{ for }i=1,\ldots,N,\label{eq: rigidity}
\end{equation}
where the $o\left(1\right)$ terms tend to zero $\mathbb{P}$-a.s.
uniformly in $i$ (see e.g. Theorem 2.9 \cite{BenaychKnowlesLecsOnLocSemicircleLaw}).

For instance from the fact the eigenvalue of largest magnitude is
of order $\sqrt{N}$ one can deduce that
\begin{equation}
\sup_{\sigma\in\mathbb{R}^{N}:\left|\sigma\right|\le1}\left|H_{N}\left(\sigma\right)\right|\le cN\text{ and }\sup_{\sigma\in\mathbb{R}^{N}:\left|\sigma\right|\le1}\left|\nabla H_{N}\left(\sigma\right)\right|\le cN,\label{eq: max}
\end{equation}
for all $N$ large enough, almost surely. The latter implies that
\begin{equation}
\left|H_{N}\left(\sigma^{1}\right)-H_{N}\left(\sigma^{2}\right)\right|\le cN\left|\sigma^{1}-\sigma^{2}\right|\text{ for all }\sigma^{i}\in\mathbb{R}^{N},\left|\sigma^{i}\right|\le1,i=1,2.\label{eq: Lipschitz}
\end{equation}

We let $E^{M}$ denote the uniform measure on the unit sphere of $\mathbb{R}^{M}$.
When $M=N$ we drop the superscript and write $E$. If $\mathcal{U}$
is a linear subspace of $\mathbb{R}^{N}$ we let $E^{\mathcal{U}}$
denote the uniform measure on the unit sphere of $\mathbb{R}^{N}$
intersected with $\mathcal{U}$.

The surface area of the $N$-dimensional sphere of radius $r$ is
$\frac{2\pi^{\frac{N}{2}}}{\Gamma\left(\frac{N}{2}\right)}r^{N-1}$,
and for any unit vector $v$ the inner product $\sigma\cdot v$ has
a density under $E$ given by
\begin{equation}
E\left[\sigma\cdot v=dx\right]=\frac{1}{\sqrt{\pi}}\frac{\Gamma\left(\frac{N}{2}\right)}{\Gamma\left(\frac{N-1}{2}\right)}\left(1-x^{2}\right)^{\frac{N-3}{2}}dx.\label{eq: one inner prod density}
\end{equation}
More generally for any linear subspace $\mathcal{U}\subset\mathbb{R}^{N}$
of dimension $M$ the the projection $\tilde{\sigma}$ of $\sigma$
onto $\mathcal{U}$ has density
\begin{equation}
E\left[d\tilde{\sigma}\right]=\frac{1}{\pi^{\frac{M}{2}}}\frac{\Gamma\left(\frac{N}{2}\right)}{\Gamma\left(\frac{N-M}{2}\right)}\left(1-\left|\tilde{\sigma}\right|^{2}\right)^{\frac{N-M-2}{2}}d\tilde{\sigma},\label{eq: high dim density}
\end{equation}
with respect to the standard Lebesgue measure on $\mathbb{R}^{N}$
restricted to $\mathcal{U}$.

\section{\label{sec: LB}Lower bound}

In this section we show the following lower bound for the free energy.
\begin{prop}
\label{prop: 2-spin TAP LB}For $\beta,h,h_{1},h_{2},\ldots$ as in
Theorem \ref{thm: main intro II} one has 
\begin{equation}
F_{N}\left(\beta,h_{N}\right)\ge\frac{1}{N}\sup_{m\in\mathbb{R}^{N}:\left|m\right|<1,\beta\left(m\right)\le\frac{1}{\sqrt{2}}}H_{TAP}\left(m\right)+o\left(1\right),\label{eq: 2-spin TAP LB}
\end{equation}
where the $o\left(1\right)$ term tends to zero $\mathbb{P}$-a.s.
\end{prop}

We prove this by noting that the partition function is certainly larger
than the integral of $e^{\beta H_{N}\left(\sigma\right)+N\sigma\cdot h_{N}}$
over a slice $\left\{ \sigma:\left|\sigma\cdot m-\left|m\right|^{2}\right|<\varepsilon\right\} $
for any $m$ inside the unit ball and $\varepsilon>0$. On this slice
we recenter the spins 
\[
\hat{\sigma}=\sigma-m,
\]
and use the decomposition 
\begin{equation}
H_{N}\left(\sigma\right)=H_{N}\left(m\right)+\nabla H_{N}\left(m\right)\cdot\hat{\sigma}+H_{N}\left(\hat{\sigma}\right),\label{eq: decomposition-1}
\end{equation}
which holds deterministically, to note that the integral over the
slice is essentially the partition function of a $2$-spin Hamiltonian
on an $N-1$-dimensional sphere of radius $1-\left|m\right|^{2}$
with mean $\beta H_{N}\left(m\right)$ and external field $\beta\nabla H_{N}\left(m\right)+Nh_{N}$.
By further restricting the integral to a subspace where the external
field vanishes the Onsager correction term $\frac{1}{2}\beta^{2}\left(1-\left|m\right|^{2}\right)^{2}$
of the TAP free energy arises as the free energy of the partition
function of this recentered Hamiltonian without external field. Plefka's
condition arises as the condition that the recentered Hamiltonian
is in high temperature.

By the second moment method and concentration of measure one can show
the following.
\begin{lem}
\label{lem: High temp}It holds that
\begin{equation}
\sup_{\beta\in\left[0,\frac{1}{\sqrt{2}}\right]}\left|\frac{1}{N}\log E\left[\exp\left(\beta H_{N}\left(\sigma\right)\right)\right]-\frac{\beta^{2}}{2}\right|\to0,\quad\ensuremath{\mathbb{P}}-\ensuremath{a.s.}\label{eq: high temp uniform in beta}
\end{equation}
\end{lem}

It will be important to consider the partition function restricted
to the intersection of the unit sphere with a hyperplane of dimension
$N-2$ (or $N-1$). The next lemma shows that (\ref{eq: high temp uniform in beta})
remains true uniformly over all such restrictions. Recall that $E^{\left\langle u,v\right\rangle ^{\bot}}$
denotes the uniform measure on the unit sphere in the subspace $\left\langle u,v\right\rangle ^{\bot}$
perpendicular to $u$ and $v$.
\begin{lem}
\label{lem: High temp in two slice}We have
\begin{equation}
\sup_{\beta\in\left[0,\frac{1}{\sqrt{2}}\right],u,v\in\mathbb{R}^{N}}\left|\frac{1}{N}\log E^{\left\langle u,v\right\rangle ^{\bot}}\left[\exp\left(\beta H_{N}\left(\sigma\right)\right)\right]-\frac{\beta^{2}}{2}\right|\to0,\quad\mathbb{P}\text{-a.s.}\label{eq: high temp in two slice}
\end{equation}
\end{lem}

\begin{proof}
Recall that $H_{N}\left(\sigma\right)=\sqrt{N}\sigma^{T}S_{N}\sigma$
where $S_{N}$ is a real symmetric matrix. For any $u,v\in\mathbb{R}^{N}$
that are linearly independent, let $w_{1},\ldots,w_{N}$ be an orthonormal
basis such that $\left\langle u,v\right\rangle =\left\langle w_{N-1},w_{N}\right\rangle $,
and let $A$ be the top left $\left(N-2\right)\times\left(N-2\right)$
minor of $S_{N}$ when written in basis $w_{1},\ldots,w_{N}$. For
$\sigma\in\left\langle u,v\right\rangle ^{\bot}$ we have $H_{N}\left(\sigma\right)=\sqrt{N}\tilde{\sigma}^{T}A\tilde{\sigma}$
where $\tilde{\sigma}=\left(\sigma_{1},\ldots,\sigma_{N-2}\right)\in\mathbb{R}^{N}$.
Let $\sqrt{N}a_{1},\ldots,\sqrt{N}a_{N-2}$ be the eigenvalues of
$A$. Then 
\begin{equation}
E^{\left\langle u,v\right\rangle ^{\bot}}\left[\exp\left(\beta H_{N}\left(\sigma\right)\right)\right]=E^{N-2}\left[\exp\left(N\beta\sum_{i=1}^{N-2}a_{i}\sigma_{i}^{2}\right)\right].\label{eq: perpendicular}
\end{equation}
Let $B$ be the top left $\left(N-2\right)\times\left(N-2\right)$
minor of $S_{N}$ when written in the standard basis and let $\sqrt{N}b_{1},\ldots,\sqrt{N}b_{N-2}$
be its eigenvalues. Note that $H_{N-2}\left(\sigma\right)=\sqrt{N-2}\sigma^{T}B\sigma$
for $\sigma\in\mathbb{R}^{N-2}$, and by (\ref{eq: high temp uniform in beta})
with $N-2$ in place of $N$ we have
\begin{equation}
E^{N-2}\left[\exp\left(\sqrt{N}\beta\sigma^{T}B\sigma\right)\right]=e^{N\left(\frac{\beta^{2}}{2}+o\left(1\right)\right)},\label{eq: N-2 FE}
\end{equation}
where the $o\left(1\right)$ term tends to zero almost surely. Also
\begin{equation}
E^{N-2}\left[\exp\left(\sqrt{N}\beta\sigma^{T}B\sigma\right)\right]=E^{N-2}\left[\exp\left(N\beta\sum_{i=1}^{N-2}b_{i}\sigma_{i}^{2}\right)\right].\label{eq: N-2}
\end{equation}

Let $\theta_{1}^{N},\ldots,\theta_{N}^{N}$ be the eigenvalues of
$S_{N}$. By the eigenvalue interlacing inequality (see e.g. Exercise
1.3.14 \cite{TaoTopicsInRMT})
\[
\theta_{i}^{N}\le a_{i},b_{i}\le\theta_{i+2}^{N}\text{ for }i=1,\ldots,N-2,
\]
so by (\ref{eq: rigidity}) we have 
\[
\sup_{i=1,\ldots,N}\left|a_{i}-b_{i}\right|\to0\text{ a.s., as }N\to\infty.
\]
Therefore
\[
\sup_{\sigma\in S_{N-2}}\left|\beta\sum_{i=1}^{N-2}a_{i}\sigma_{i}^{2}-\beta\sum_{i=1}^{N-2}b_{i}\sigma_{i}^{2}\right|\to0\text{ a.s., as }N\to\infty,
\]
so from (\ref{eq: perpendicular}), (\ref{eq: N-2 FE}) and (\ref{eq: N-2})
it follows that
\[
E^{\left\langle u,v\right\rangle ^{\bot}}\left[\exp\left(\beta H_{N}\left(\sigma\right)\right)\right]=e^{N\frac{\beta^{2}}{2}\left(1+o\left(1\right)\right)},
\]
uniformly over all linearly independent $u,v$, where the $o\left(1\right)$
terms tend to zero almost surely. The above argument but with $\left(N-1\right)\times\left(N-1\right)$
minors easily extends this to $u$ and $v$ that are linearly dependent.
This proves (\ref{eq: high temp in two slice}).
\end{proof}
We can now prove the lower bound Proposition \ref{prop: 2-spin TAP LB}.
\begin{proof}[Proof of Proposition \ref{prop: 2-spin TAP LB}]
For any $m$ and $\sigma$, recenter the spins $\sigma$ around $m$
by letting $\hat{\sigma}=\sigma-m$. Recentering the Hamiltonian (see
(\ref{eq: decomposition-1})) and the external field one obtains 
\begin{equation}
\beta H_{N}\left(\sigma\right)+N\sigma\cdot h_{N}=\beta H_{N}\left(m\right)+Nm\cdot h_{N}+Nh^{m}\cdot\hat{\sigma}+\beta H_{N}\left(\hat{\sigma}\right),\label{eq: decomp with ext field}
\end{equation}
where
\begin{equation}
h^{m}=\frac{\beta}{N}\nabla H_{N}\left(m\right)+h_{N},\label{eq: eff ext field}
\end{equation}
is the effective external field after recentering. Note that by our
assumption $\left|h_{N}\right|=h$ and (\ref{eq: max}) we have that
for $N$ large enough 
\begin{equation}
\left|h^{m}\right|\le c,\label{eq: effective ext field bound}
\end{equation}
for a constant $c$ depending only on $\beta$ and $h$.

Fix an $m\in\mathbb{R}^{N}$ with $\left|m\right|<1$. Let $v_{1},v_{2}$
be basis vectors of an arbitrary two dimensional linear subspace of
$\mathbb{R}^{N}$ that contains $m$ and $h^{m}$. For $\varepsilon>0$
to be fixed later consider 
\begin{equation}
A=\left\{ \sigma:\hat{\sigma}\cdot v_{i}\in\left(-\varepsilon,\varepsilon\right),i=1,2\right\} .\label{eq: Adef}
\end{equation}
Note that for $\sigma\in A$ 
\begin{equation}
\left|\hat{\sigma}\cdot m\right|\le c\varepsilon\text{ and }\left|\hat{\sigma}\cdot h^{m}\right|\le c\varepsilon,\label{eq: inner prod}
\end{equation}
(the latter constant depends on the one in (\ref{eq: effective ext field bound}))
and
\begin{equation}
\left|\hat{\sigma}\right|^{2}=\left|\sigma\right|^{2}-\left|m\right|^{2}-2\hat{\sigma}\cdot m=1-\left|m\right|^{2}+O\left(\varepsilon\right).\label{eq: recentered magnitude}
\end{equation}

Certainly we have
\[
Z_{N}\left(\beta,h_{N}\right)\ge E\left[1_{A}\exp\left(\beta H_{N}\left(\sigma\right)+N\sigma\cdot h_{N}\right)\right].
\]
Rewriting in terms of $\hat{\sigma}$ and using (\ref{eq: decomp with ext field})
and the second inequality of (\ref{eq: inner prod}) the right hand-side
can be bounded below by
\begin{equation}
\exp\left(\beta H_{N}\left(m\right)+Nm\cdot h_{N}-c\varepsilon N\right)E\left[1_{A}\exp\left(\beta H_{N}\left(\hat{\sigma}\right)\right)\right].\label{eq: thing}
\end{equation}
Let $\gamma\sigma^{\bot}$ be the projection of $\hat{\sigma}$ onto
the hyperplane $\left\langle v_{1},v_{2}\right\rangle ^{\bot}$, where
$\sigma^{\bot}$ is a unit vector and $\gamma\in\mathbb{R}$ is the
magnitude of the projection. From (\ref{eq: Adef}) we have for $\sigma\in A$
\begin{equation}
\left|\hat{\sigma}-\gamma\sigma^{\bot}\right|\le c\varepsilon,\label{eq: dist}
\end{equation}
so that by (\ref{eq: recentered magnitude})
\begin{equation}
\gamma^{2}=1-\left|m\right|^{2}+O\left(\varepsilon\right).\label{eq: gamma2 magnitude}
\end{equation}
Using (\ref{eq: dist}) and (\ref{eq: max})-(\ref{eq: Lipschitz})
we have
\[
H_{N}\left(\hat{\sigma}\right)=H_{N}\left(\gamma\sigma^{\bot}\right)+O\left(\varepsilon N\right),
\]
and by (\ref{eq: scaling}), (\ref{eq: max}) and (\ref{eq: gamma2 magnitude})
\[
H_{N}\left(\gamma\sigma^{\bot}\right)=\gamma^{2}H_{N}\left(\sigma^{\bot}\right)=\left(1-\left|m\right|^{2}\right)H_{N}\left(\sigma^{\bot}\right)+O\left(\varepsilon N\right).
\]
This gives that (\ref{eq: thing}) is at least
\begin{equation}
\exp\left(\beta H_{N}\left(m\right)+Nm\cdot h_{N}-c\varepsilon N\right)E\left[1_{A}\exp\left(\beta\left(1-\left|m\right|^{2}\right)H_{N}\left(\sigma^{\bot}\right)\right)\right].\label{eq: intermediate}
\end{equation}
Now $\sigma^{\bot}$ is independent of $\sigma\cdot m,\sigma\cdot h^{m}$
under $E$, and is uniform on the unit sphere intersected with $\left\langle v_{1},v_{2}\right\rangle ^{\bot}$.
Therefore (\ref{eq: intermediate}) in fact equals
\[
\exp\left(\beta H_{N}\left(m\right)+Nm\cdot h_{N}-c\varepsilon N\right)E\left[A\right]E^{\left\langle v_{1},v_{2}\right\rangle ^{\bot}}\left[\exp\left(\beta\left(1-\left|m\right|^{2}\right)H_{N}\left(\sigma\right)\right)\right].
\]
Using (\ref{eq: high dim density}) with $M=2$ and (\ref{eq: recentered magnitude})
and it holds that
\[
E\left[A\right]\ge Nc\varepsilon^{2}\left(1-\left|m\right|^{2}-c\varepsilon\right)^{\frac{N-4}{2}},
\]
and setting e.g. $\varepsilon=\frac{1}{\sqrt{N}}$ this equals
\[
\exp\left(\frac{N}{2}\log\left(1-\left|m\right|^{2}\right)+o\left(N\right)\right).
\]
Thus $Z_{N}$ is at least
\[
\begin{array}{l}
\exp\left(\beta H_{N}\left(m\right)+Nm\cdot h_{N}+\frac{N}{2}\log\left(1-\left|m\right|^{2}\right)+o\left(N\right)\right)\\
\times E^{\left\langle v_{1},v_{2}\right\rangle ^{\bot}}\left[\exp\left(\beta\left(1-\left|m\right|^{2}\right)H_{N}\left(\sigma\right)\right)\right],
\end{array}
\]
for any $m$ with $\left|m\right|<1$, where the error term is $o\left(N\right)$
uniformly in $m$, almost surely.

By Lemma \ref{lem: High temp in two slice} this is in turn at least
\begin{equation}
\exp\left(\beta H_{N}\left(m\right)+Nm\cdot h_{N}+\frac{N}{2}\log\left(1-\left|m\right|^{2}\right)+N\frac{\beta^{2}}{2}\left(1-\left|m\right|^{2}\right)^{2}+o\left(N\right)\right),\label{eq: final}
\end{equation}
provided 
\begin{equation}
\beta\left(1-\left|m\right|^{2}\right)\le\frac{1}{\sqrt{2}},\text{ i.e. }\beta\left(m\right)\le\frac{1}{\sqrt{2}},\label{eq: plefka in proof}
\end{equation}
where the error term is $o\left(N\right)$ almost surely, uniformly
in $m$ that satisfy (\ref{eq: plefka in proof}). Since (\ref{eq: final})
equals $\exp\left(H_{TAP}\left(m\right)+o\left(N\right)\right)$ the
claim (\ref{eq: 2-spin TAP LB}) follows.
\end{proof}

\section{\label{sec: UB}Upper bound}

In this section we prove the following upper bound on the free energy.
\begin{prop}
\label{prop: 2-spin TAP UB}For $\beta,h,h_{1},h_{2},\ldots$ as in
Theorem \ref{thm: main intro II} one has 
\begin{equation}
F_{N}\left(\beta,h_{N}\right)\le\frac{1}{N}\sup_{m\in\mathbb{R}^{N}:\left|m\right|<1,\beta\left(m\right)\le\frac{1}{\sqrt{2}}}H_{TAP}\left(m\right)+o\left(1\right),\label{eq: 2-spin TAP UB}
\end{equation}
where the $o\left(1\right)$ term tends to zero $\mathbb{P}$-a.s.
\end{prop}

As for the lower bound, our proof is based on considering the Hamiltonian
recentered around certain $m$-s inside the unit ball. However, for
an upper bound we are not free to simply restrict the integral in
the partition function to slices around an $m$ and ignore the complement.
Neither can we further restrict the integral inside the slice to a
space where the effective external field vanishes. Lastly we can not
ignore slices for which Plefka's condition is not satisfied.

We get around these issues by constructing a low-dimensional subspace
$\mathcal{M}_{N}$ of $m$-s, such that the recentered Hamiltonian
restricted to the space of configurations perpendicular to $\mathcal{M}_{N}$
has almost vanishing external field for \emph{any} $m\in\mathcal{M}_{N}$,
without further restriction. Because the dimension of $\mathcal{M}_{N}$
is $o\left(N\right)$ we are able to use the Laplace method to upper
bound the free energy by a $\sup$ of the free energy contribution
of each of these restricted Hamiltonians. Lastly a coarse-graining
of the recentered Hamiltonian gives a sequence of approximations to
the free energy of the restricted Hamiltonians in a form that allows
to show that the supremum must be attained at an $m$ that satisfies
Plefka's condition. 

\subsection{\label{sec: diag}Diagonalization}

To prove the upper bound Proposition \ref{prop: 2-spin TAP UB} we
are obliged to make stronger use the diagonalized Hamiltonian 
\begin{equation}
N\sum_{i=1}^{N}\theta_{i}^{N}\sigma_{i}^{2},\label{eq: diagonalized Hamiltonian}
\end{equation}
and the semi-circle law. Let 
\begin{equation}
\tilde{h}_{N}\text{ be the vector }h_{N}\text{ written in the diagonalizing basis of the matrix }S_{N}.\label{eq: h tilde}
\end{equation}
By rotational symmetry we have
\[
F_{N}\left(\beta,h_{N}\right)=\frac{1}{N}\log E\left[\exp\left(N\beta\sum_{i=1}^{N}\theta_{i}^{N}\sigma_{i}^{2}+N\tilde{h}_{N}\cdot\sigma\right)\right].
\]
For convenience we also replace the diagonalized Hamiltonian (\ref{eq: diagonalized Hamiltonian})
by its deterministic counterpart
\[
\tilde{H}_{N}\left(\sigma\right)=N\sum_{i=1}^{N}\theta_{i/N}\sigma_{i}^{2},
\]
where each random eigenvalue $\theta_{i}^{N}$ is replaced by its
deterministic typical position $\theta_{i/N}$ (recall (\ref{eq: classical location})).
The error made is controlled by (\ref{eq: rigidity}), giving
\begin{equation}
\lim_{N\to\infty}\frac{1}{N}\sup_{\sigma:\left|\sigma\right|=1}\left|N\sum_{i=1}^{N}\theta_{i}^{N}\sigma_{i}^{2}-\tilde{H}_{N}\left(\sigma\right)\right|=0,\quad\mathbb{P}-a.s.\label{eq: error made}
\end{equation}
Let
\[
\tilde{F}_{N}\left(\beta,h_{N}\right)=\frac{1}{N}\log E\left[\exp\left(N\beta\sum_{i=1}^{N}\theta_{i/N}\sigma_{i}^{2}+N\tilde{h}_{N}\cdot\sigma\right)\right].
\]
 and let
\begin{equation}
\tilde{H}_{TAP}\left(m\right)=\beta\tilde{H}_{N}\left(m\right)+Nm\cdot\tilde{h}_{N}+\frac{N}{2}\log\left(1-\left|m\right|^{2}\right)+N\frac{\beta^{2}}{2}\left(1-\left|m\right|^{2}\right)^{2}.\label{eq: HTAP tilde}
\end{equation}
By (\ref{eq: error made}) the upper bound Proposition \ref{prop: 2-spin TAP UB}
follows from the following deterministic bound.
\begin{prop}
\label{prop: UB diag deterministic}For $\beta,h,h_{1},h_{2},\ldots$
as in Theorem \ref{thm: main intro II} one has
\begin{equation}
\tilde{F}_{N}\left(\beta,h_{N}\right)\le\frac{1}{N}\sup_{m\in\mathbb{R}^{N}:\left|m\right|<1,\beta\left(m\right)\le\frac{1}{\sqrt{2}},}\tilde{H}_{TAP}\left(m\right)+o\left(1\right).\label{eq: suff to show diag deterministic}
\end{equation}
\end{prop}

The rest of this section is devoted to the proof of Proposition \ref{prop: UB diag deterministic}.

\subsection{\label{sec: high temp no ext field UB}Free energy of coarse-grained
Hamiltonian without external field}

We will approximate $\tilde{H}_{N}\left(\sigma\right)$ by a coarse-grained
Hamiltonian where the $\theta_{i/N}$ are replaced by a bounded number
of distinct coefficients. For such a Hamiltonian it will be straight-forward
to bound the free energy using the Laplace method. To this end consider
for each $K\ge2$ equally spaced numbers $x_{1},\ldots,x_{K}$ in
$\left[-\sqrt{2},\sqrt{2}\right]$, so that, 
\[
-\sqrt{2}=x_{1}<x_{2}<\ldots<x_{K}=\sqrt{2}-\frac{2\sqrt{2}}{K}\text{ and }x_{k+1}-x_{k}=\frac{2\sqrt{2}}{K},
\]
and a partition $I_{1},\ldots,I_{K}$ of $\left\{ 1,\ldots,N\right\} $
given by
\begin{equation}
I_{k}=\left\{ i:x_{k}\le\theta_{i/N}<x_{k+1}\right\} ,k=1,\ldots,K-1\text{ and }I_{K}=\left\{ i:x_{K}\le\theta_{i/N}\right\} .\label{eq: Ik def}
\end{equation}
Let
\begin{equation}
\sigma_{\left[k\right]}^{2}=\sum_{i\in I_{k}}\sigma_{i}^{2}\text{ and }\mu_{k}=\frac{\left|I_{k}\right|}{N}.\label{eq: muk def}
\end{equation}

The next lemma gives the density of the vector $\left(\sigma_{\left[1\right]}^{2},\ldots\sigma_{\left[K-1\right]}^{2}\right)$
under $E$.
\begin{lem}
\label{lem: density}The $E$-distribution of the vector $\left(\sigma_{\left[1\right]}^{2},\ldots\sigma_{\left[K-1\right]}^{2}\right)$
has a density on $\mathbb{R}^{K-1}$ with respect to Lebesgue measure
given by
\begin{equation}
\Gamma\left(\frac{N}{2}\right)\prod_{k=1}^{K}\frac{\rho_{k}^{\frac{\left|I_{k}\right|-2}{2}}}{\Gamma\left(\frac{\left|I_{k}\right|}{2}\right)}1_{A}d\rho_{1}\ldots d\rho_{K-1},\label{eq: density}
\end{equation}
where we write $\rho_{K}=1-\rho_{1}-\ldots-\rho_{K-1}$ and $A=\left\{ \rho_{1},\ldots,\rho_{K-1}\ge0,\rho_{1}+\ldots+\rho_{K-1}\le1\right\} $.
\end{lem}

\begin{proof}
One can sample the random variable $\sigma$ with law $E$ by sampling
from the standard Gaussian distribution on $\mathbb{R}^{N}$ and normalizing
the result. Therefore $\left(\sigma_{\left[1\right]}^{2},\ldots\sigma_{\left[K-1\right]}^{2}\right)$
has the same law as $\left(R_{1},\ldots,R_{K-1}\right),$ where
\[
R_{i}=\frac{X_{i}}{X_{1}+\ldots+X_{K}},i\le K,
\]
the $X_{i}$ are independent, and $X_{i}$ has the $\chi^{2}$-distribution
with $\left|I_{k}\right|$ degrees of freedom, i.e. has density $\frac{1}{2^{\left|I_{k}\right|/2}\Gamma\left(\left|I_{k}\right|/2\right)}x_{i}^{\frac{\left|I_{k}\right|-2}{2}}e^{-\frac{x_{k}}{2}}1_{\{x_{k}\ge0\}}dx_{k}$.
We now let $Z=X_{1}+\ldots+X_{K}$ and make the change of variables
$x_{i}=z\rho_{i},i=1,\ldots,K-1$ which has Jacobian $z^{k-1}$ to
obtain that $\left(R_{1},\ldots,R_{K-1},Z\right)$ has density
\[
\begin{array}{l}
1_{A}1_{\left\{ z\ge0\right\} }\left(\prod_{k=1}^{K}\frac{1}{2^{\left|I_{k}\right|/2}\Gamma\left(\frac{\left|I_{k}\right|}{2}\right)}\left(z\rho_{i}\right)^{\frac{\left|I_{k}\right|-2}{2}}e^{-\frac{z\rho_{k}}{2}}\right)z^{K-1}d\rho_{1}\ldots d\rho_{k-1}dz.\\
=\left(1_{A}\prod_{k=1}^{K}\frac{\rho_{i}^{\frac{\left|I_{k}\right|-2}{2}}}{\Gamma\left(\frac{\left|I_{k}\right|}{2}\right)}\right)\left(\frac{1}{2^{N/2}}1_{\left\{ z\ge0\right\} }z^{\frac{N-2}{2}}e^{-\frac{z}{2}}dz\right)d\rho_{1}\ldots d\rho_{k-1}.
\end{array}
\]
Since
\[
\int\frac{1}{2^{N/2}}z^{\frac{N-2}{2}}e^{-\frac{z}{2}}dz=\Gamma\left(\frac{N}{2}\right),
\]
integrating out $z$ to get the marginal of $\left(R_{1},\ldots,R_{K-1}\right)$
one obtains (\ref{eq: density}).
\end{proof}
We first show the following variational principle for the free energy
of the coarse-grained Hamiltonians in the absence of an external field.
\begin{lem}
\textup{\label{lem: finite number of weights FE}}For all $C>0$ we
have uniformly in $0<\beta\le C$\textup{ , large enough $K$ and
$N\ge c\left(K\right)$ that}
\begin{equation}
\begin{array}{l}
\frac{1}{N}\log E\left[\exp\left(N\beta\sum_{k=1}^{K}x_{k}\sigma_{\left[k\right]}^{2}\right)\right]\\
={\displaystyle \sup_{0\le f_{k},f_{1}+\ldots+f_{K}=1}}\left\{ \beta\sum_{k=1}^{K}x_{k}f_{k}+\frac{1}{2}\sum\mu_{k}\log\frac{f_{k}}{\mu_{k}}\right\} +O\left(\frac{K^{3}\log N}{N}\right).
\end{array}\label{eq: finite number of weights FE}
\end{equation}
\end{lem}

\begin{proof}
By Lemma \ref{lem: density} the integral $E\left[\exp\left(N\beta\sum_{k=1}^{K}x_{k}\sigma_{\left[k\right]}^{2}\right)\right]$
equals
\begin{equation}
\frac{\Gamma\left(\frac{N}{2}\right)}{\prod_{k=1}^{K}\Gamma\left(\frac{\left|I_{k}\right|}{2}\right)}\int_{\left[0,1\right]^{K-1}}1_{A}\exp\left(N\left\{ \beta\sum_{k=1}^{K}x_{k}\rho_{k}+\sum_{k=1}^{K}\frac{1}{2}\left(\mu_{k}-\frac{2}{N}\right)\log\rho_{k}\right\} \right)d\rho_{1}\ldots d\rho_{K-1}.\label{eq: with density}
\end{equation}
By the Laplace method the integral in (\ref{eq: with density}) is
at most
\begin{equation}
\exp\left(N\left\{ {\displaystyle \sup_{0\le f_{k},f_{1}+\ldots+f_{K}=1}}\left\{ \beta\sum_{k=1}^{K}x_{k}f_{k}+\frac{1}{2}\sum\left(\mu_{k}-\frac{2}{N}\right)\log f_{k}\right\} \right\} \right).\label{eq: bla2}
\end{equation}
To get rid of the nuisance term $\frac{2}{N}$ we use the following
ad-hoc argument. For any maximizer of the $\sup$ the value must exceed
$-\sqrt{2}\beta-\frac{1}{2}\log K$, since one obtains at least this
by setting $f_{1}=\ldots=f_{K}=\frac{1}{K}$. Note that $\mu_{k}\ge\frac{c}{K^{3/2}}$
for all $k,K$, by (\ref{eq: semi-circle law}), so for $N\ge c\left(K\right)$
also $\mu_{k}-\frac{2}{N}\ge\frac{c}{K^{3/2}}$. Assume now that $f_{k}\le e^{-K^{2}}$
from some $k$. Then $\beta\sum_{k=1}^{K}x_{k}f_{k}+\frac{1}{2}\sum\left(\mu_{k}-\frac{2}{N}\right)\log f_{k}\le\sqrt{2}\beta-c\frac{K^{2}}{K^{3/2}}<-\sqrt{2}\beta-\frac{1}{2}\log K$,
for $K$ large enough. So for $K$ large enough and $N\ge c\left(K\right)$
any maximizer in the $\sup$ above must satisfy $f_{k}\ge e^{-K^{2}}$.
But for such $f_{k}$ the nuisance term contributes at most $\frac{K^{3}}{N}$.
Therefore (\ref{eq: bla2}) equals
\[
\exp\left(N\left\{ {\displaystyle \sup_{0\le f_{k},f_{1}+\ldots+f_{K}=1}}\left\{ \beta\sum_{k=1}^{K}x_{k}f_{k}+\frac{1}{2}\sum\mu_{k}\log f_{k}\right\} +O\left(\frac{K^{3}}{N}\right)\right\} \right).
\]

Using the bounds $\Gamma\left(x\right)\asymp\sqrt{2\pi x}\left(x/e\right)^{x}$
for $x\ge\frac{1}{2}$ and $1\le\prod_{k=1}^{K}\left|I_{k}\right|\le N^{K},$
one sees that $\frac{1}{N}\log$ of the factor multiplying the integral
in (\ref{eq: with density}) equals
\[
\begin{array}{l}
\frac{1}{N}\log\frac{\Gamma\left(\frac{N}{2}\right)}{\prod_{k=1}^{K}\Gamma\left(\frac{\left|I_{k}\right|}{2}\right)}\\
=\frac{1}{N}\log\frac{\left(\frac{N}{2}\right)^{\frac{N}{2}}}{\prod_{k=1}^{K}\left(\frac{\left|I_{k}\right|}{2}\right)^{\frac{\left|I_{k}\right|}{2}}}+\frac{1}{N}\log\frac{\sqrt{2\pi\frac{N}{2}}}{\prod_{k=1}^{K}\sqrt{2\pi\frac{\left|I_{k}\right|}{2}}}+O\left(\frac{K}{N}\right)\\
=-\frac{1}{2}\sum\mu_{k}\log\mu_{k}+O\left(\frac{K\log N}{N}\right).
\end{array}
\]
This completes the proof.
\end{proof}
The variational problem on the bottom line of (\ref{eq: finite number of weights FE})
can be solved. To state the result let 
\begin{equation}
g_{K}\left(\lambda\right)=\sum_{k=1}^{K}\frac{\mu_{k}}{\lambda-x_{k}}.\label{eq: gkdef}
\end{equation}
For all $\beta>0$ there is a unique $\lambda_{K}\left(\beta\right)>x_{K}$
such that
\begin{equation}
g_{K}\left(\lambda_{K}\left(\beta\right)\right)=2\beta.\label{eq: lambdakdef}
\end{equation}
Let
\[
h_{K}\left(\lambda\right)=\sum_{k=1}^{K}\mu_{k}\log\left(\lambda-x_{k}\right),
\]
and
\begin{equation}
\mathcal{F}_{K}\left(\beta\right)=\beta\lambda_{K}\left(\beta\right)-\frac{1}{2}-\frac{1}{2}\log\left(2\beta\right)-\frac{1}{2}h_{K}\left(\lambda_{K}\left(\beta\right)\right).\label{eq: F K def}
\end{equation}
The next lemma shows that $\mathcal{F}_{K}\left(\beta\right)$ is
the supremum in the variational problem from (\ref{eq: finite number of weights FE}).
\begin{lem}
\label{lem: Opt problem solution}For each $K$ and $\beta>0$ we
have
\begin{equation}
{\displaystyle \sup_{0\le f_{k},f_{1}+\ldots+f_{K}=1}}\left\{ \beta\sum_{k=1}^{K}x_{k}f_{k}+\frac{1}{2}\sum\mu_{k}\log\frac{f_{k}}{\mu_{k}}\right\} =\mathcal{\mathcal{F}}_{K}\left(\beta\right).\label{eq: supremum}
\end{equation}
\end{lem}

\begin{proof}
Since the quantity being maximized tends to $-\infty$ if $f_{k}\to0$
for some $k$ there must be a global maximum satisfying $f_{k}>0$
for all $k$. Using Lagrange multipliers to solve the constrained
optimization problem one considers
\[
\mathcal{L}\left(f_{1},\ldots,f_{k},\lambda\right)=\beta\sum_{k=1}^{K}x_{k}f_{k}+\frac{1}{2}\sum\mu_{k}\log\frac{f_{k}}{\mu_{k}}+\tilde{\lambda}\left(\sum_{k=1}^{K}f_{k}-1\right).
\]
If $\left(f_{1},\ldots,f_{K}\right)$ is a global maximum then there
must be a $\tilde{\lambda}$ such that $\left(f_{1},\ldots,f_{K},\tilde{\lambda}\right)$
is a critical point of $\mathcal{L}$. The critical point equations
of $\mathcal{L}$ read 
\[
\beta x_{k}-\frac{1}{2}\frac{1}{f_{k}}+\tilde{\lambda}=0,k=1,\ldots,K,\text{ and }\sum_{k=1}^{K}f_{k}-1=0.
\]
The first $K$ equations are equivalent to
\begin{equation}
f_{k}=\frac{1}{2\beta}\frac{\mu_{k}}{\lambda-x_{k}},k=1,\ldots,K,\label{eq: maximizer}
\end{equation}
where we reparameterized $\lambda=-\beta\tilde{\lambda}$. Therefore
for some $\lambda>x_{K}$ it holds that
\[
\sum_{k=1}^{K}f_{k}=1,
\]
for the $f_{k}$ in (\ref{eq: maximizer}) and these $f_{k}$ maximize
(\ref{eq: supremum}). Inspection of (\ref{eq: gkdef})-(\ref{eq: lambdakdef})
reveal that $\lambda=\lambda_{K}\left(\beta\right)$ is the unique
such $\lambda$. When $f_{k}$ take the form in (\ref{eq: maximizer})
then
\begin{equation}
\sum_{k=1}^{K}x_{k}f_{k}=\sum_{k=1}^{K}x_{k}\frac{1}{2\beta}\frac{\mu_{k}}{\lambda-x_{k}}=\frac{1}{2\beta}\left(\lambda\sum_{k=1}^{K}\frac{\mu_{k}}{\lambda-x_{k}}-1\right)=\lambda-\frac{1}{2\beta},\label{eq: contributing energy}
\end{equation}
and
\[
\frac{1}{2}\sum\mu_{k}\log\frac{f_{k}}{\mu_{k}}=\frac{1}{2}\sum\mu_{k}\log\frac{1}{2\beta}\frac{1}{\lambda-x_{k}}=-\frac{1}{2}\log\left(2\beta\right)-\frac{1}{2}h_{K}\left(\lambda\right).
\]
Therefore the value of the quantity being maximized at the unique
maximizer is the right-hand side of (\ref{eq: supremum}).
\end{proof}
Note that Lemmas \ref{lem: finite number of weights FE} and \ref{lem: Opt problem solution}
show that the free energy of the coarse-grained Hamiltonians has no
phase transition for any finite $K$. Also those lemmas and the bound
\begin{equation}
\sum_{i=1}^{N}\theta_{i/N}\sigma_{i}^{2}=\sum_{k=1}^{K}x_{k}\sigma_{\left[k\right]}^{2}+O\left(K^{-1}\right),\label{eq: blocking}
\end{equation}
imply that $\mathcal{F}_{K}\left(\beta\right)$ is an approximation
of the free energy of the Hamiltonian $N\sum_{i=1}^{N}\theta_{i/N}\sigma_{i}^{2}$.
\begin{lem}
\label{lem: High temperature from truncated}For all $C>0$ and $K\ge2$
we have
\begin{equation}
\limsup_{N\to\infty}\sup_{\beta\in\left[0,C\right]}\left|\frac{1}{N}\log E^{N}\left[\exp\left(\beta N\sum_{i=1}^{N}\theta_{i/N}\sigma_{i}^{2}\right)\right]-\mathcal{F}_{K}\left(\beta\right)\right|\le\frac{c}{K}.\label{eq: high temperature}
\end{equation}
\end{lem}

We now investigate the behavior of $\mathcal{F}_{K}\left(\beta\right)$
as $K\to\infty$. Let

\begin{equation}
g\left(\lambda\right)=\int_{-\sqrt{2}}^{\sqrt{2}}\frac{\mu\left(x\right)}{\lambda-x}dx\text{ for }\lambda\ge\sqrt{2}.\label{eq: gdef}
\end{equation}
By standard estimates for Riemann sums 
\begin{equation}
\lim_{K\to\infty}g_{K}\left(\lambda\right)=g\left(\lambda\right)\text{ for }\lambda>\sqrt{2}.\label{eq: g conv}
\end{equation}
The integral can be computed explicitly, and in fact
\[
g\left(\lambda\right)=\lambda-\sqrt{\lambda^{2}-2}.
\]
Note that $g\left(\sqrt{2}\right)=\sqrt{2}$. If $\beta\le\frac{1}{\sqrt{2}}$
there is a unique $\lambda\left(\beta\right)\ge\sqrt{2}$ such that
$g\left(\lambda\left(\beta\right)\right)=2\beta$. In fact 
\begin{equation}
\lambda\left(\beta\right)=\frac{1}{\sqrt{2}}\left(\sqrt{2}\beta+\frac{1}{\sqrt{2}\beta}\right)\text{ for }\beta\le\frac{1}{\sqrt{2}}.\label{eq: lambda beta}
\end{equation}
The convergence (\ref{eq: g conv}) implies that 
\begin{equation}
\lim_{K\to\infty}\lambda_{K}\left(\beta\right)=\lambda\left(\beta\right)\text{ for }\beta<\frac{1}{\sqrt{2}}.\label{eq: lambda conv}
\end{equation}
Also define
\[
h\left(\lambda\right)=\int_{-\sqrt{2}}^{\sqrt{2}}\mu\left(x\right)\log\left(\lambda-x\right)dx\text{ for }\lambda\ge\sqrt{2},
\]
which can be computed explicitly as 
\begin{equation}
h\left(\lambda\right)=\frac{\lambda^{2}}{2}-\frac{1}{2}-\frac{\lambda\sqrt{\lambda^{2}-2}}{2}+\log\left(\frac{\lambda+\sqrt{\lambda^{2}-2}}{2}\right).\label{eq: h lambda}
\end{equation}
By the convergence of the Riemann sum
\begin{equation}
\lim_{K\to\infty}h_{K}\left(\lambda\right)=h\left(\lambda\right)\text{ for }\lambda>\sqrt{2}.\label{eq: h conv}
\end{equation}

Define
\begin{equation}
\mathcal{F}\left(\beta\right)=\beta\lambda\left(\beta\right)-\frac{1}{2}-\frac{1}{2}\log\left(2\beta\right)-\frac{1}{2}h\left(\lambda\left(\beta\right)\right),\beta\in\left[0,\frac{1}{\sqrt{2}}\right].\label{eq: F def}
\end{equation}
Using the identities (\ref{eq: lambda beta}) and (\ref{eq: h lambda}),
this expression simplifies to
\begin{equation}
\mathcal{F\left(\beta\right)}=\frac{\beta^{2}}{2}\text{ for }\beta\in\left[0,\frac{1}{\sqrt{2}}\right].\label{eq: F beta squared over two}
\end{equation}
Also it follows from (\ref{eq: lambda conv}) and (\ref{eq: h conv})
and the monotonicity of $h_{K}\left(\lambda\right)$ that
\begin{equation}
\lim_{K\to\infty}\mathcal{F}_{K}\left(\beta\right)=\mathcal{F}\left(\beta\right)\text{ if }\beta<\frac{1}{\sqrt{2}}.\label{eq:F pointwise conv}
\end{equation}
A posteriori it is clear that for $\beta>\frac{1}{\sqrt{2}}$ the
function $\mathcal{F}_{K}\left(\beta\right)$ converges to the low-temperature
free energy of the Hamiltonian $H_{N}\left(\sigma\right)$ without
external field, but this is not a step in the proof of our main results,
but rather a consequence. 

In the proof of Proposition \ref{prop: UB diag deterministic} at
the end of the next section we will use the two lemmas that now follow
to rule out $m$ that do not satisfy Plefka's condition. First note
that $g_{K}\left(\lambda\right),\lambda_{K}\left(\beta\right),h_{K}\left(\lambda\right)$
and thus $\mathcal{\mathcal{F}}_{K}\left(\beta\right)$ are all continuous
and differentiable. We have the following identity.
\begin{lem}
\label{lem: FE K deriv}For all $\beta>0$
\begin{equation}
\mathcal{F}_{K}^{'}\left(\beta\right)=\lambda_{K}\left(\beta\right)-\frac{1}{2\beta}.\label{eq: derivative}
\end{equation}
\end{lem}

\begin{proof}
This follows from the definition (\ref{eq: F K def}) and the equalities
$h_{K}^{'}=g_{K}$ and $g_{K}\left(\lambda\left(\beta\right)\right)=2\beta$.
\end{proof}

\begin{lem}
\label{lem: lambda large is high temp}For all $K\ge2$ there is an
$\varepsilon\in\left(0,\frac{2\sqrt{2}}{K}\right)$ such that
\[
\lambda_{K}\left(\beta\right)\ge\sqrt{2}-\varepsilon\implies\beta\le\frac{1}{\sqrt{2}}.
\]
\end{lem}

\begin{proof}
We may set $\sqrt{2}-\varepsilon=\sqrt{2}-\lambda_{K}\left(\frac{1}{\sqrt{2}}\right)$
since
\[
\lambda_{K}\left(\beta\right)\ge\lambda_{K}\left(\frac{1}{\sqrt{2}}\right)\implies\beta\le\frac{1}{\sqrt{2}},
\]
and
\[
x_{K}<\lambda_{K}\left(\frac{1}{\sqrt{2}}\right)<\lambda\left(\frac{1}{\sqrt{2}}\right)=\sqrt{2},
\]
where the second inequality follows because $g_{K}\left(\lambda\right)<g\left(\lambda\right)$
for $\lambda\ge\sqrt{2}$ (see (\ref{eq: Ik def}), (\ref{eq: muk def}),
(\ref{eq: gdef}), (\ref{eq: gkdef})) and $g_{K}\left(\lambda\right)$
is decreasing in $\lambda$, implying that the solution to $g_{K}\left(\lambda\right)=\sqrt{2}$
must occur for $\lambda<\sqrt{2}$.
\end{proof}
Lemma \ref{lem: FE K deriv} also allows us to strengthen the pointwise
convergence (\ref{eq:F pointwise conv}) to uniform convergence.
\begin{lem}
\label{lem: F K uniform conv}We have
\begin{equation}
\lim_{K\to\infty}\sup_{\beta\in\left[0,\frac{1}{\sqrt{2}}\right]}\left|\mathcal{F}_{K}\left(\beta\right)-\mathcal{F}\left(\beta\right)\right|=0.\label{eq: F K uniform conv}
\end{equation}
\end{lem}

\begin{proof}
The $\mathcal{F}_{K}\left(\beta\right)$ are increasing in $\beta$
(because the left-hand side of (\ref{eq: supremum}) is) and $\mathcal{F}\left(\beta\right)$
is increasing in $\beta\in\left[0,\frac{1}{\sqrt{2}}\right]$ and
uniformly continuous (recall (\ref{eq: F beta squared over two})).
This implies that the pointwise convergence (\ref{eq:F pointwise conv})
can be strengthened to uniform convergence on $\left[0,\frac{1}{\sqrt{2}}-\delta\right]$
for any $\delta>0,$ i.e. 
\[
\lim_{K\to\infty}\sup_{\beta\in\left[0,\frac{1}{\sqrt{2}}-\delta\right]}\left|\mathcal{F}_{K}\left(\beta\right)-\mathcal{F}\left(\beta\right)\right|=0.
\]
For any $\delta>0$ we have
\[
\sup_{\beta\in\left[\frac{1}{\sqrt{2}}-\delta,\frac{1}{\sqrt{2}}\right]}\left|\mathcal{F}\left(\beta\right)-\mathcal{F}\left(\beta-\delta\right)\right|\le c\delta,
\]
and for any $\delta\in\left(0,\frac{1}{2\sqrt{2}}\right)$ (say) we
have uniformly in $K$ that
\[
\sup_{\beta\in\left[\frac{1}{\sqrt{2}}-\delta,\frac{1}{\sqrt{2}}\right]}\left|\mathcal{F}_{K}\left(\beta\right)-\mathcal{F}_{K}\left(\beta-\delta\right)\right|\le c\delta,
\]
by (\ref{eq: derivative}) ($\lambda_{K}\left(\beta\right)$ is decreasing
in $\beta$ and $\lambda_{K}\left(\frac{1}{2}\right)$ is bounded
by (\ref{eq: lambda conv})). Thus for such $\delta$ also
\[
\sup_{\beta\in\left[\frac{1}{\sqrt{2}}-\delta,\frac{1}{\sqrt{2}}\right]}\left|\mathcal{F}_{K}\left(\beta\right)-\mathcal{F}\left(\beta\right)\right|\le c\delta.
\]
Thus
\[
\lim_{K\to\infty}\sup_{\beta\in\left[0,\frac{1}{\sqrt{2}}\right]}\left|\mathcal{F}_{K}\left(\beta\right)-\mathcal{F}\left(\beta\right)\right|\le c\delta,
\]
for all $\delta\in\left(0,\frac{1}{2\sqrt{2}}\right)$, so the claim
(\ref{eq: F K uniform conv}) follows.
\end{proof}

\subsection{\label{sec: MN space}Making the external field after recentering
vanish}

As for the lower bound, an important step in the proof of the upper
bound is to recenter the Hamiltonian around an $m\in\mathbb{R}^{N}$
which yields an effective external field $\beta\frac{1}{N}\nabla\tilde{H}_{N}\left(m\right)+\tilde{h}_{N}$
(cf. (\ref{eq: decomp with ext field})-(\ref{eq: eff ext field})
and (\ref{eq: recenter})). In this section the main goal is Lemma
\ref{lem: low dim space that gets rid of ext field}, which constructs
a low-dimensional subspace $\mathcal{M}_{N}\subset\mathbb{R}^{N}$,
such that if we recenter around any $m\in\mathcal{M}_{N}$ the effective
external field is again (almost) contained in $\mathcal{M}_{N}$ (so
that if we restrict to the space perpendicular to $\mathcal{M}_{N}$,
the effective external field after recentering almost vanishes). Its
use will be an important step in the proof of the upper bound in the
next subsection. 

The construction in the proof of Lemma \ref{lem: low dim space that gets rid of ext field}
will involve taking the span of a vector (close to) $\tilde{h}_{N}$
iterated under the the map $\frac{1}{N}\nabla\tilde{H}_{N}$. For
this the next lemma will be needed, whose claim (\ref{eq: proj bound})
says says that after applying the map $\frac{1}{N}\nabla\tilde{H}_{N}$
to a vector $v\in\mathbb{R}^{N}$ a large number of times, the resulting
vector will be almost completely contained in the space spanned by
the eigenvectors associated to the eigenvalues of largest magnitude.
Let $\Pi^{A}$ denote the projection onto a subspace $A\subset\mathbb{R}^{N}$. 
\begin{lem}
\label{lem: iterate bound}For any $\varepsilon>0$, $N\ge1$, $v\in\mathbb{R}^{N}$
with $v_{N}\ne0$ and $k\ge1$ it holds that 
\begin{equation}
\frac{\left|\Pi^{\left\langle e_{i}:\left|\theta_{i/N}\right|<\sqrt{2}-\varepsilon\right\rangle }\left(\frac{1}{N}\nabla\tilde{H}_{N}\right)^{k}v\right|}{\left|\left(\frac{1}{N}\nabla\tilde{H}_{N}\right)^{k}v\right|}\le\sqrt{N}\left|v\right|v_{N}^{-1}e^{-c\varepsilon k}.\label{eq: proj bound}
\end{equation}
\end{lem}

\begin{proof}
Denote the matrix $\frac{1}{N}\nabla\tilde{H}_{N}$ by $D$. Note
that $D$ is diagonal and $D_{ii}=2\theta_{i/N}$. Thus for any $v\in\mathbb{R}^{N}$
we have
\[
\left(D^{k}v\right)_{i}=\left(2\theta_{i/N}\right)^{k}v_{i}.
\]
Now for $v$ such that $v_{N}\ne0$ and $i$ such that $\left|\theta_{i/N}\right|<\sqrt{2}-\varepsilon$
we have
\[
\frac{\left|\left(D^{k}v\right)_{i}\right|}{\left|D^{k}v\right|}=\frac{\left|\left(2\theta_{i/N}\right)^{k}v_{i}\right|}{\sqrt{\sum_{i=1}^{N}\left(2\theta_{i/N}\right)^{2k}v_{i}^{2}}}\le\frac{\left|\theta_{i/N}\right|^{k}\left|v\right|}{\sqrt{2}^{k}v_{N}}\le v_{N}^{-1}\left|v\right|\left(1-c\varepsilon\right)^{k}\le\left|v\right|v_{N}^{-1}e^{-c\varepsilon k}.
\]
By taking the square and summing over the at most $N$ indices $i$
such that $\left|\theta_{i/N}\right|<\sqrt{2}-\varepsilon$ the claim
(\ref{eq: proj bound}) follows.
\end{proof}
The next lemma is a weak bound on the proportion of all eigenvalues
have magnitude close to the maximal magnitude.
\begin{lem}
\label{lem: number of theta bound}For all $N\ge1$ and $\varepsilon>0$
\begin{equation}
\left|\left\{ i:\left|\theta_{i/N}\right|\ge\sqrt{2}-\varepsilon\right\} \right|\le c\varepsilon N.\label{eq: number of theta bound}
\end{equation}
 
\end{lem}

\begin{proof}
This follows for instance by noting that $\theta_{\left(i+1\right)/N}-\theta_{i/N}\ge cN^{-1}$,
which is a consequence of the definition (\ref{eq: classical location})
of $\theta_{i/N}$ and the bound $\int_{\sqrt{2}-\varepsilon}^{\sqrt{2}}\mu\left(x\right)dx\le\varepsilon\sup_{x}\mu\left(x\right)\le c\varepsilon$.
\end{proof}

We now construct the subspaces $\mathcal{M}_{N}$. Recall that $\frac{1}{N}\nabla\tilde{H}_{N}$
is a linear map (cf. (\ref{eq: gradient linear})) and that the standard
basis vectors $e_{i}$ are its eigenvectors, so the span of any set
of basis vectors in invariant under $\frac{1}{N}\nabla\tilde{H}_{N}$.
\begin{lem}
\label{lem: low dim space that gets rid of ext field}Let $\beta,h,h_{1},h_{2},\ldots$
be as in Theorem \ref{thm: main intro II}. There exists a sequence
of linear spaces $\mathcal{\mathcal{M}}_{1},\mathcal{M}_{2},\ldots$
such that $\mathcal{\mathcal{M}}_{N}\subset\mathbb{R}^{N}$,
\begin{equation}
\dim\left(\mathcal{\mathcal{M}}_{N}\right)=\lfloor N^{3/4}\rfloor,\label{eq: dim}
\end{equation}
and $\mathcal{M}_{N}$ is approximately invariant under the map $m\to\beta\frac{1}{N}\nabla\tilde{H}_{N}\left(m\right)+\tilde{h}_{N}$
in the sense that
\begin{equation}
\lim_{N\to\infty}\sup_{m\in\mathcal{M}_{N},\left|m\right|\le1}\left|\Pi^{\mathcal{M}_{N}^{\bot}}\left(\beta\frac{1}{N}\nabla\tilde{H}_{N}\left(m\right)+\tilde{h}_{N}\right)\right|=0.\label{eq: ext field vanish}
\end{equation}
\end{lem}

\begin{proof}
We will construct $\mathcal{M}_{N}$ so that it contains a vector
$\bar{h}_{N}$ close to $\tilde{h}_{N}$ and is approximately invariant
under the map $\frac{1}{N}\nabla\tilde{H}_{N}\left(m\right)$. More
precisely let
\begin{equation}
\bar{h}_{N,i}=\tilde{h}_{N,i}\text{\,for }i\le N-1\text{ and }\bar{h}_{N,N}=\begin{cases}
\tilde{h}_{N,N} & \text{ if }\left|\tilde{h}_{N,N}\right|\ge\frac{1}{N},\\
\frac{1}{N} & \text{ if }\left|\tilde{h}_{N,N}\right|<\frac{1}{N},
\end{cases}\label{eq: h bar}
\end{equation}
so that
\[
\left|\bar{h}_{N}-\tilde{h}_{N}\right|\le\frac{1}{N}\text{ and }\left|\bar{h}_{N,N}\right|\ge\frac{1}{N}.
\]
We construct $\mathcal{M}_{N}$ so that
\begin{equation}
\bar{h}_{N}\in\mathcal{M}_{N},\label{eq: M1}
\end{equation}
and $\mathcal{M}_{N}$ is almost invariant under $\frac{1}{N}\nabla\tilde{H}_{N}$
in the sense that
\begin{equation}
\lim_{N\to\infty}\sup_{m\in\mathcal{M}_{N},\left|m\right|\le1}\left|\Pi^{\mathcal{M}_{N}^{\bot}}\left(\frac{1}{N}\nabla\tilde{H}_{N}\left(m\right)\right)\right|=0.\label{eq: M2}
\end{equation}
Since $\left|\Pi^{\mathcal{M}_{N}^{\bot}}\tilde{h}_{N}\right|\le\left|\Pi^{\mathcal{M}_{N}^{\bot}}\bar{h}_{N}\right|+\frac{1}{N}=\frac{1}{N}$
and
\[
\left|\Pi^{\mathcal{M}_{N}^{\bot}}\left(\beta\frac{1}{N}\nabla\tilde{H}_{N}\left(m\right)+\tilde{h}_{N}\right)\right|\le\beta\left|\Pi^{\mathcal{M}_{N}^{\bot}}\left(\frac{1}{N}\nabla\tilde{H}_{N}\left(m\right)\right)\right|+\left|\Pi^{\mathcal{M}_{N}^{\bot}}\tilde{h}_{N}\right|,
\]
this implies (\ref{eq: ext field vanish}). Furthermore it suffices
to construct $\mathcal{M}_{N}$ so that
\begin{equation}
\dim\mathcal{M}_{N}\le\lfloor N^{3/4}\rfloor,\label{eq: size bound}
\end{equation}
since by adding arbitrary basis vectors $e_{i}$ (which are invariant
under $\frac{1}{N}\nabla\tilde{H}_{N}$) to the span of $\mathcal{M}_{N}$
one can ensure $\dim\mathcal{M}_{N}=\lfloor N^{3/4}\rfloor$ while
maintaining (\ref{eq: M1}) and (\ref{eq: M2}).

To ensure (\ref{eq: M2}) we will let $\mathcal{M}_{N}$ contain the
span of a sufficient number of vectors
\[
\bar{h}_{N}^{k}=\left(\frac{1}{N}\nabla\tilde{H}_{N}\right)^{k}\bar{h}_{N},k\ge0,
\]
and basis vectors $e_{i}$ belonging to the eigenvalues $\theta_{i/N}$
of largest magnitude. Let
\[
\hat{h}_{N}^{k}=\frac{\bar{h}_{N}^{k}}{\left|\bar{h}_{N}^{k}\right|},
\]
be normalized vectors and construct
\[
\mathcal{M}_{N}=\left\langle \hat{h}_{N}^{0},\ldots,\hat{h}_{N}^{V-1},e_{j}:j\in J\right\rangle ,
\]
for
\[
V=\sqrt{N}\left(\log N\right)^{2},
\]
and
\[
J=\left\{ j:\left|\theta_{j/N}\right|\ge\sqrt{2}-N^{-1/2}\right\} .
\]
Clearly (\ref{eq: M1}) holds since $\mathcal{M}_{N}$ contains $\hat{h}_{N}^{0}=\bar{h}_{N}/\left|\bar{h}_{N}\right|$.
Using Lemma \ref{lem: number of theta bound} it also holds that
\[
\dim\mathcal{M}_{N}\le V+\left|J\right|\le\sqrt{N}\left(\log N\right)^{2}+c\sqrt{N},
\]
which implies (\ref{eq: size bound}).

To check (\ref{eq: M2}), note that for any $m\in\mathcal{M}_{N}$
with $\left|m\right|\le1$ we may decompose $m$ as 
\begin{equation}
m=\sum_{k=0}^{V-1}\alpha_{k}\hat{h}_{N}^{k}+\sum_{i\in J}\gamma_{i}e_{i},\label{eq: decomposition}
\end{equation}
for some $\alpha_{0},\ldots,\alpha_{V-1}\in\mathbb{R},\gamma_{i}\in\mathbb{R},i\in J$,
where we first set set $\alpha_{V-1}=m\cdot\hat{h}_{N}^{V-1}$, to
ensure that 
\begin{equation}
\left|\alpha_{V-1}\right|\le1,\label{eq: last alpha}
\end{equation}
before picking the other coefficients in the decomposition. Thus
\[
\begin{array}{lcl}
\frac{1}{N}\nabla\tilde{H}_{N}\left(m\right) & = & \sum_{k=0}^{V-1}\alpha_{k}\frac{1}{N}\nabla\tilde{H}_{N}\left(\hat{h}_{N}^{k}\right)+\sum_{i\in J}\gamma_{i}\frac{1}{N}\nabla\tilde{H}_{N}\left(e_{i}\right)\\
 & = & \sum_{k=0}^{V-1}\alpha_{k}\frac{\left|\bar{h}_{N}^{k+1}\right|}{\left|\bar{h}_{N}^{k}\right|}\hat{h}_{N}^{k+1}+\sum_{i\in J}\gamma_{i}2\theta_{i/N}e_{i}.
\end{array}
\]
Therefore
\[
\Pi^{\mathcal{M}_{N}^{\bot}}\left(\frac{1}{N}\nabla\tilde{H}_{N}\left(m\right)\right)=\alpha_{V-1}\frac{\left|\bar{h}_{N}^{V}\right|}{\left|\bar{h}_{N}^{V-1}\right|}\Pi^{\mathcal{M}_{N}^{\bot}}\left(\hat{h}_{N}^{V}\right).
\]
Note that since $\|\frac{1}{N}\nabla\tilde{H}_{N}\|=2\theta_{0}=2\theta_{1}=2\sqrt{2}$
we have $\left|\bar{h}_{N}^{V}\right|\le2\sqrt{2}\left|\bar{h}_{N}^{V-1}\right|$
and so using also (\ref{eq: last alpha})
\[
\left|\Pi^{\mathcal{M}_{N}^{\bot}}\frac{1}{N}\nabla\tilde{H}_{N}\left(m\right)\right|\le c\left|\Pi^{\mathcal{M}_{N}^{\bot}}\hat{h}_{N}^{V}\right|.
\]

The point of (\ref{eq: h bar}) was to ensure that $\left|\bar{h}_{N,N}\right|\ge\frac{1}{N}$,
so that Lemma \ref{lem: iterate bound} applies to $\bar{h}_{N}$.
With $\varepsilon=N^{-1/2}$ it gives that
\begin{equation}
\left|\Pi^{\mathcal{M}_{N}^{\bot}}\hat{h}_{N}^{V}\right|=\frac{\left|\Pi^{\mathcal{M}_{N}^{\bot}}\bar{h}_{N}^{V}\right|}{\left|\bar{h}_{N}^{V}\right|}\le\frac{\left|\Pi^{\left\langle e_{i}:\theta_{i/N}<\sqrt{2}-\varepsilon\right\rangle }\bar{h}_{N}^{V}\right|}{\left|\bar{h}_{N}^{V}\right|}\le c\sqrt{N}\left|\bar{h}_{N}\right|\bar{h}_{N,N}^{-1}e^{-c\left(\log N\right)^{2}}=o\left(1\right),\label{eq: projection}
\end{equation}
so (\ref{eq: M2}) follows. Since we have constructed $\mathcal{M}_{N}$
satisfying (\ref{eq: M1}), (\ref{eq: M2}) and (\ref{eq: size bound})
the proof is complete.
\end{proof}
We will need a version of Lemma \ref{lem: High temperature from truncated}
where we integrate over the subspace perpendicular to $\mathcal{M}_{N}$.
\begin{lem}
\label{lem: High temp truncated subspace}For any $C>0$ and $K>0$
\[
\limsup_{N\to\infty}{\displaystyle \sup_{\beta\in\left[0,C\right]}\left|\frac{1}{N}\log E^{\mathcal{U}_{N}}\left[\exp\left(\beta N\sum_{i=1}^{N}\theta_{i/N}\sigma_{i}^{2}\right)\right]-\mathcal{F}_{K}\left(\beta\right)\right|}\le\frac{c}{K},
\]
where $\mathcal{U}_{N}=\mathcal{M}_{N}^{\bot}$ and $\mathcal{M}_{N},N\ge1,$
is the sequence of subspaces from Lemma \ref{lem: low dim space that gets rid of ext field}.
\end{lem}

\begin{proof}
This follows from Lemma \ref{lem: High temperature from truncated}
similarly to how Lemma \ref{lem: High temp in two slice} follows
from Lemma \ref{lem: High temp}. Let $M=\lfloor N^{3/4}\rfloor$.
Consider an orthonormal basis of $\mathbb{R}^{N-M}$ such that the
space $\mathcal{U}_{N}$ is spanned by the first $N-M$ basis vectors
and let $A$ be the $\left(N-M\right)\times\left(N-M\right)$ minor
of the matrix $D$ which in the standard basis is diagonal with $D_{ii}=\theta_{i/N}$.
The eigenvalues $a_{1},\ldots,a_{N-M}$ of $A$ satisfy $a_{i}=\theta_{i/N}+o\left(1\right)=\theta_{i/\left(N-M\right)}+o\left(1\right)$
by the eigenvalue interlacing inequality, so that an estimate for
$E^{\mathcal{U}_{N}}\left[\cdot\right]$ follows from Lemma \ref{lem: High temperature from truncated}
with $N-M$ in place of $N$.
\end{proof}

\subsection{\label{sec: proof of UB-1}Proof of upper bound}

We are now ready to complete the proof of the upper bound. Define
a modified TAP free energy by replacing the Onsager correction $\frac{1}{2}\beta^{2}\left(1-\left|m\right|^{2}\right)^{2}$
by $\mathcal{F}_{K}\left(\beta\left(1-\left|m\right|^{2}\right)\right)$
to obtain
\[
\tilde{H}_{TAP}^{K}\left(m\right)=\beta\tilde{H}_{N}\left(m\right)+Nm\cdot\tilde{h}_{N}+\frac{N}{2}\log\left(1-\left|m\right|^{2}\right)+N\mathcal{F}_{K}\left(\beta\left(1-\left|m\right|^{2}\right)\right).
\]
We have the following version of the upper bound Proposition \ref{prop: UB diag deterministic}
with $\tilde{H}_{TAP}^{K}\left(m\right)$ in place of $\tilde{H}_{TAP}\left(m\right)$
and without a Plefka condition.
\begin{prop}
\label{prop: 2-spin TAP UB K}For all $K\ge2$ and $\beta,h,h_{1},h_{2},\ldots$
as in Theorem \ref{thm: main intro II} we have 
\begin{equation}
\tilde{F}_{N}\left(\beta,h_{N}\right)\le\frac{1}{N}\sup_{m\in\mathbb{R}^{N}:\left|m\right|<1}\tilde{H}_{TAP}^{K}\left(m\right)+\frac{c}{K},\label{eq: 2-spin TAP UB K}
\end{equation}
for large enough $N$.
\end{prop}

\begin{proof}
Let $\mathcal{M}_{N}$ be the space from Lemma \ref{lem: low dim space that gets rid of ext field}
and let
\[
\mathcal{U}_{N}=\mathcal{M}_{N}^{\bot}.
\]
Let $M=\lfloor N^{3/4}\rfloor$. For any $\sigma\in\mathbb{R}^{N}$
let $m$ be the projection of $\sigma$ onto $\mathcal{M}_{N}$ and
$\hat{\sigma}=\sigma-m\in\mathcal{U}_{N}$. Recentering the Hamiltonian
around $m$ (cf. (\ref{eq: decomp with ext field})-(\ref{eq: eff ext field}))
we have that
\begin{equation}
\begin{array}{l}
E\left[\exp\left(\beta\tilde{H}_{N}\left(\sigma\right)+N\tilde{h}_{N}\cdot\sigma\right)\right]\\
=E\left[\exp\left(N\beta\tilde{H}_{N}\left(m\right)+N\tilde{h}_{N}\cdot m+N\left(\beta\frac{1}{N}\nabla\tilde{H}_{N}\left(m\right)+\tilde{h}_{N}\right)\cdot\hat{\sigma}+\beta\tilde{H}_{N}\left(\hat{\sigma}\right)\right)\right].
\end{array}\label{eq: recenter}
\end{equation}
Lemma \ref{lem: low dim space that gets rid of ext field} implies
that 
\[
\lim_{N\to\infty}\sup_{m\in\mathcal{M}_{N}}\sup_{\hat{\sigma}\in\mathcal{M}_{N}^{\bot},\left|\hat{\sigma}\right|\le1}\left(\beta\frac{1}{N}\nabla\tilde{H}_{N}\left(m\right)+\tilde{h}_{N}\right)\cdot\hat{\sigma}=0,
\]
so the effective external field vanishes and (\ref{eq: recenter})
is at most
\begin{equation}
e^{o\left(N\right)}E\left[\exp\left(N\beta\tilde{H}_{N}\left(m\right)+N\tilde{h}_{N}\cdot m+\beta\tilde{H}_{N}\left(\hat{\sigma}\right)\right)\right].\label{eq: no ext field}
\end{equation}
Note that the the $E\left[\cdot|m\right]$-law of $\hat{\sigma}$
is the uniform distribution on sphere in the subspace $\mathcal{U}_{N}$
of radius $\sqrt{1-\left|m\right|^{2}}$. Thus using also (\ref{eq: scaling})
this equals
\begin{equation}
E\left[\exp\left(N\beta\tilde{H}_{N}\left(m\right)+N\tilde{h}_{N}\cdot m\right)E^{\mathcal{U}_{N}}\left[\beta\left(1-\left|m\right|^{2}\right)\tilde{H}_{N}\left(\sigma\right)\right]\right].\label{eq: subspace}
\end{equation}
By Lemma \ref{lem: High temp truncated subspace} this is at most
\begin{equation}
E\left[\exp\left(N\beta\tilde{H}_{N}\left(m\right)+N\tilde{h}_{N}\cdot m+\mathcal{F}_{K}\left(\beta\left(1-\left|m\right|^{2}\right)\right)\right)\right]e^{o\left(N\right)+\frac{c}{K}}.\label{eq: subpace FK}
\end{equation}
Using (\ref{eq: high dim density}) the $E$-integral equals
\begin{equation}
a_{N}\int_{m:\left|m\right|<1}\left(1-\left|m\right|^{2}\right)^{\frac{N-M-2}{2}}\exp\left(N\beta H_{N}\left(m\right)+N\tilde{h}_{N}\cdot m+N\mathcal{F}_{K}\left(\beta\left(1-\left|m\right|^{2}\right)\right)\right)dm,\label{eq: integral}
\end{equation}
where $a_{N}=\frac{1}{\pi^{\frac{N-M}{2}}}\frac{\Gamma\left(\frac{N}{2}\right)}{\Gamma\left(\frac{M}{2}\right)}$
and the integral is $M$-dimensional against Lebesgue measure on $\mathcal{M}_{N}$.
This equals
\[
a_{N}\int_{m:\left|m\right|<1}\exp\left(\tilde{H}_{TAP}^{K}\left(m\right)+\left(M+2\right)\left|\log\left(1-\left|m\right|^{2}\right)\right|\right)dm,
\]
and by the Laplace method is bounded above by
\[
a_{N}\exp\left(\sup_{m:\left|m\right|<1}\left\{ \tilde{H}_{TAP}^{K}\left(m\right)+\left(M+2\right)\left|\log\left(1-\left|m\right|^{2}\right)\right|\right\} \right)\int_{m:\left|m\right|<1}dm.
\]
The $M$-dimensional Lebesgue integral $\int_{m:\left|m\right|<1}dm$
is the volume of the unit ball in dimension $M$ which equals $\frac{\pi^{\frac{M}{2}}}{\Gamma\left(\frac{M}{2}+1\right)}=O\left(1\right)$,
and $\log a_{N}=o\left(N\right)$, so this is at most
\begin{equation}
\exp\left(o\left(N\right)+\sup_{m:\left|m\right|<1}\left\{ \tilde{H}_{TAP}^{K}\left(m\right)+\left(M+2\right)\left|\log\left(1-\left|m\right|^{2}\right)\right|\right\} \right).\label{eq: almost}
\end{equation}
To get rid of the nuisance term involving $M+2$, note that there
is a $\delta$ depending only on $\beta$ and $h$ such that the supremum
is always achieved for $\left|m\right|<1-\delta$, since all terms
in the supremum not involving $\log$ are bounded by $cN$. Thus (\ref{eq: almost})
is at most 
\begin{equation}
\exp\left(o\left(N\right)+\sup_{m:\left|m\right|<1}\tilde{H}_{TAP}^{K}\left(m\right)+cM\right).\label{eq: last}
\end{equation}
This is then also an upper bound for (\ref{eq: integral}), which
shows that (\ref{eq: subpace FK}) and therefore $\tilde{F}_{N}\left(\beta,h_{N}\right)$
is bounded by $\exp\left(\sup_{m:\left|m\right|<1}\tilde{H}_{TAP}^{K}\left(m\right)+o\left(N\right)+\frac{c}{K}\right)$.
This implies (\ref{eq: 2-spin TAP UB K})
\end{proof}
We can now prove the upper bound Proposition \ref{prop: UB diag deterministic}
for free energy of the diagonal and deterministic Hamiltonian $\tilde{H}_{N}\left(\sigma\right)$,
by showing that the $\sup$ in (\ref{eq: 2-spin TAP UB K}) is bounded
above by that in (\ref{eq: suff to show diag deterministic}).
\begin{proof}[Proof of Proposition \ref{prop: UB diag deterministic}]
Fix $K\ge2$. For any $N\ge1$, consider the variational problem
\[
\sup_{m\in\mathbb{R}^{N}:\left|m\right|<1}\tilde{H}_{TAP}^{K}\left(m\right).
\]
Any local maximum $m$ of $\tilde{H}_{TAP}^{K}\left(m\right)$ must
satisfy
\[
\nabla\tilde{H}_{TAP}^{K}\left(m\right)=0,
\]
and
\begin{equation}
\nabla^{2}\tilde{H}_{TAP}^{K}\left(m\right)\text{ is negative semi-definite}.\label{eq: neg semi def}
\end{equation}
The gradient of $\tilde{H}_{TAP}^{K}$ is
\[
\nabla\tilde{H}_{TAP}^{K}\left(m\right)=\beta\nabla\tilde{H}_{N}\left(m\right)+N\tilde{h}_{N}-Nm\left(\frac{1}{1-\left|m\right|^{2}}+2\beta\mathcal{F}_{K}^{'}\left(\beta\left(1-\left|m\right|^{2}\right)\right)\right).
\]
By Lemma \ref{lem: FE K deriv} we have for all $m$ that
\[
\nabla\tilde{H}_{TAP}^{K}\left(m\right)=\beta\nabla\tilde{H}_{N}\left(m\right)+N\tilde{h}_{N}-N2\beta m\lambda_{K}\left(\beta\left(1-\left|m\right|^{2}\right)\right).
\]
Thus the Hessian $\nabla^{2}\tilde{H}_{N}^{K}\left(m\right)$ equals
\[
\beta\nabla^{2}\tilde{H}_{N}\left(m\right)-N2\beta I\lambda_{K}\left(\beta\left(1-\left|m\right|^{2}\right)\right)+4\beta^{2}Nmm^{T}\lambda_{K}^{'}\left(\beta\left(1-\left|m\right|^{2}\right)\right).
\]
For any local maximum $m$ let 
\[
A=\frac{1}{2N}\nabla^{2}\tilde{H}_{N}\left(m\right)-I\lambda_{K}\left(\beta\left(1-\left|m\right|^{2}\right)\right),
\]
and
\[
B=2m\left(m\right)^{T}\lambda_{K}^{'}\left(\beta\left(1-\left|m\right|^{2}\right)\right).
\]
Since $B$ is of rank one, the second largest eigenvalue $a_{N-1}$
of $A$ is bounded above by the largest eigenvalue of $A+B$. The
latter matrix is the Hessian at $m$ multiplied by a positive scalar,
so all its eigenvalues are non-positive. Thus $a_{N-1}\le0$.  Furthermore
$\frac{1}{2N}\nabla^{2}\tilde{H}_{N}\left(m\right)=\frac{1}{N}D$
where $D$ is the diagonal matrix with $D_{ii}=\theta_{i/N}$, so
the eigenvalues of $A$ are $\theta_{i/N}-\lambda_{K}\left(\beta\left(1-\left|m\right|^{2}\right)\right)$.
This shows that 
\[
\lambda_{K}\left(\beta\left(1-\left|m\right|^{2}\right)\right)\ge\theta_{1-\frac{1}{N}},
\]
at $m$ which are local maxima. Since
\[
\theta_{1-1/N}=\sqrt{2}+o\left(1\right),
\]
it follows from Lemma \ref{lem: lambda large is high temp} that we
must have for such $m$
\[
\beta\left(1-\left|m\right|^{2}\right)\le\frac{1}{\sqrt{2}},\text{ that is }\beta\left(m\right)\le\frac{1}{\sqrt{2}},
\]
(provided $N$ large enough depending on $K$), and by Lemma \ref{lem: F K uniform conv}
\[
\mathcal{F}_{K}\left(\beta\left(1-\left|m\right|^{2}\right)\right)\le\frac{1}{2}\beta^{2}\left(1-\left|m\right|^{2}\right)^{2}+\varepsilon_{K},
\]
where $\lim_{K\to\infty}\varepsilon_{K}=0$. Thus from (\ref{eq: 2-spin TAP UB K})
it holds for such $N$ that
\[
\tilde{F}_{N}\left(\beta,h_{N}\right)\le\frac{1}{N}\sup_{m\in\mathbb{R}^{N}:\left|m\right|<1,\beta\left(m\right)\le\frac{1}{\sqrt{2}}}\tilde{H}_{TAP}\left(m\right)+\varepsilon_{K}+\frac{c}{K}.
\]
We have shown that
\[
\limsup_{N\to\infty}\left\{ \frac{1}{N}\sup_{m\in\mathbb{R}^{N}:\left|m\right|<1,\beta\left(m\right)\le\frac{1}{\sqrt{2}}}\tilde{H}_{TAP}\left(m\right)-\tilde{F}_{N}\left(\beta,h_{N}\right)\right\} \le\varepsilon_{K}+\frac{c}{K},
\]
for all $K\ge2$. Since the left-hand side is independent of $K$,
it is in fact at most $0$. This implies (\ref{eq: suff to show diag deterministic}).
\end{proof}
This also completes the proof of the main upper bound Proposition
\ref{prop: 2-spin TAP UB}. Together with the lower bound Proposition
\ref{prop: 2-spin TAP LB} this proves our main result Theorem \ref{thm: main intro II}.

\section{\label{sec: TAPvarSol}Solution of the TAP-Plefka variational problem}

In this section we prove Lemma \ref{lem: TAP var sol}. By (\ref{eq: scaling})
it follows from a result for the maximum of the Hamiltonian with external
field on the unit sphere which we now state.
\begin{lem}
For $h,h_{1},h_{2},\ldots$ as in Theorem \ref{thm: main intro II}
we have
\begin{equation}
\sup_{\sigma:\left|\sigma\right|=1}\left\{ \beta\frac{1}{N}H_{N}\left(\sigma\right)+h_{N}\cdot\sigma\right\} \to\sqrt{h^{2}+2\beta^{2}},\label{eq: GS with ext field}
\end{equation}
in probability.
\end{lem}

\begin{proof}
We work in the diagonalizing basis of $S_{N}$ and note that the left-hand
side of (\ref{eq: GS with ext field}) equals
\begin{equation}
\sup_{\sigma:\left|\sigma\right|=1}\left\{ \beta\sum_{i=1}^{N}\theta_{i/N}\sigma_{i}^{2}+\tilde{h}_{N}\cdot\sigma\right\} +o\left(1\right),\label{eq: diagonalized shell max}
\end{equation}
where, as in Section \ref{sec: diag}, $\tilde{h}_{N}$ is the vector
$h_{N}$ written in the diagonalizing basis and we have used (\ref{eq: rigidity}).
The case $h=0$ then follows trivially since $\theta_{1}=\sqrt{2}$,
so we assume in the sequel that $h>0$. For any $\lambda>\sqrt{2}$
let
\[
\sigma_{i}\left(\lambda\right)=\frac{1}{2\beta}\frac{\left(\tilde{h}_{N}\right)_{i}}{\lambda-\theta_{i/N}}.
\]
Using Lagrange multipliers the maximizer of (\ref{eq: diagonalized shell max})
can be shown to be $\sigma_{i}=\sigma_{i}\left(\lambda_{N}\right)$
where $\lambda_{N}>\sqrt{2}$ is the number such that $\sum_{i=1}^{N}\sigma_{i}^{2}\left(\lambda_{N}\right)=1$.
By rotational symmetry the $\mathbb{P}$-law of $\left(\tilde{h}_{N}\right)_{i}$
is that of a uniform random vector on $\left\{ x\in\mathbb{R}^{N}:\left|x\right|=h\right\} $.
Using this one can show that for any $\lambda>\sqrt{2}$
\[
\left|\sum_{i=1}^{N}\sigma_{i}\left(\lambda\right)^{2}-\frac{h^{2}}{2\beta}\frac{1}{N}\sum_{i=1}^{N}\frac{1}{\left(\lambda-\theta_{i/N}\right)^{2}}\right|\to0,\text{ in probability}.
\]
Also for $\lambda>\sqrt{2}$
\[
\sum_{i=1}^{N}\frac{1}{\left(\lambda-\theta_{i/N}\right)^{2}}\to\int_{-\sqrt{2}}^{\sqrt{2}}\frac{\mu\left(x\right)}{\left(\lambda-x\right)^{2}}dx=\frac{\lambda}{\sqrt{\lambda^{2}-2}}-1,
\]
and since for 
\[
\tilde{\lambda}=\sqrt{\frac{2}{1-\left(1+\frac{4\beta^{2}}{h^{2}}\right)^{-2}}},
\]
and $\lambda=\tilde{\lambda}$ we have $\lambda/\sqrt{\lambda^{2}-2}-1=2\beta/h^{2}$,
it follows that 
\[
\lambda_{N}\to\tilde{\lambda},\text{ in probability}.
\]

Similarly for any $\lambda>\sqrt{2}$ we have that
\[
\sum_{i=1}^{N}\theta_{i/N}\sigma_{i}\left(\lambda\right)^{2}\to\frac{h^{2}}{2\beta}\int_{-\sqrt{2}}^{\sqrt{2}}\frac{x}{\left(\lambda-x\right)^{2}}\mu\left(x\right)dx=\frac{h^{2}}{2\beta^{2}}\left(\frac{\lambda^{2}-1}{\sqrt{\lambda^{2}-2}}-\lambda\right),
\]
and
\[
\sum_{i=1}^{N}\left(\tilde{h}_{N}\right)_{i}\sigma_{i}\left(\lambda\right)\to\frac{h^{2}}{2\beta}\int_{-\sqrt{2}}^{\sqrt{2}}\frac{\mu\left(x\right)}{\lambda-x}dx=\frac{h^{2}}{2\beta}\left(\lambda-\sqrt{\lambda^{2}-2}\right),
\]
both in probability. This shows that
\[
\beta\sum_{i=1}^{N}\theta_{i/N}\sigma_{i}\left(\lambda_{N}\right)^{2}+\tilde{h}_{N}\cdot\sigma\left(\lambda_{N}\right)\to\frac{h^{2}}{2\beta}\left(\frac{\tilde{\lambda}^{2}-1}{\sqrt{\tilde{\lambda}^{2}-2}}-\tilde{\lambda}\right)+\frac{h^{2}}{2\beta}\left(\tilde{\lambda}-\sqrt{\tilde{\lambda}^{2}-2}\right),
\]
in probability, and the right-hand side simplifies to $\sqrt{h^{2}+2\beta^{2}}$.
\end{proof}
\bibliographystyle{plain}

\begin{thebibliography}{10}

\bibitem{AizenmanLebowitzRuelleSomeRigorousResultsOnTheSherringtonKirckpatrick}
Michael Aizenman, Joel~L Lebowitz, and David Ruelle.
\newblock {Some rigorous results on the {S}herrington-{K}irkpatrick spin glass
  model}.
\newblock {\em Communications in Mathematical Physics}, 112(1):3--20, 1987.

\bibitem{AuffingerBenArousCernyComplexityofSpinGlasses}
Antonio Auffinger, G{\'e}rard {Ben Arous}, and Ji\v{r}{\'i} \v{C}ern{\'y}.
\newblock {Random matrices and complexity of spin glasses}.
\newblock {\em Comm. Pure Appl. Math.}, 66(2):165--201, 2013.

\bibitem{AuffingerJagannathTAPequationsforcondGibbsmeasuresingeneric}
Antonio Auffinger and Aukosh Jagannath.
\newblock {{T}houless-{A}nderson-{P}almer equations for conditional {G}ibbs
  measures in the generic p-spin glass model}.
\newblock {\em arXiv preprint arXiv:1612.06359}, 2016.

\bibitem{BaikLeeFluctuationsOfTheFEofTheSphericalSK}
Jinho Baik and Ji~Oon Lee.
\newblock {Fluctuations of the free energy of the spherical
  {S}herrington--{K}irkpatrick model}.
\newblock {\em Journal of Statistical Physics}, 165(2):185--224, 2016.

\bibitem{BenaychKnowlesLecsOnLocSemicircleLaw}
Florent Benaych-Georges and Antti Knowles.
\newblock {Lectures on the local semicircle law for Wigner matrices}.
\newblock {\em arXiv preprint arXiv:1601.04055}, 2016.

\bibitem{BolthausenPrivateCommunication}
Erwin Bolthausen.
\newblock Private communication.

\bibitem{BolthausenAnIterativeConstructionOfSoloftheTAPequations}
Erwin Bolthausen.
\newblock {An iterative construction of solutions of the {TAP} equations for
  the {S}herrington-{K}irkpatrick model}.
\newblock {\em Comm. Math. Phys.}, 325(1):333--366, 2014.

\bibitem{BraggWilliamsTheEffectofThermalAgitationonAtomicArrangementInAlloys}
William~L Bragg and Evan~J Williams.
\newblock {The effect of thermal agitation on atomic arrangement in alloys}.
\newblock {\em Proceedings of the Royal Society of London. Series A, Containing
  Papers of a Mathematical and Physical Character}, 145(855):699--730, 1934.

\bibitem{ChatterjeeSpinGlassesandSteinsMethod}
Sourav Chatterjee.
\newblock {Spin glasses and {S}tein's method}.
\newblock {\em Probability theory and related fields}, 148(3):567--600, 2010.

\bibitem{ChenPanchekoOntheTAPFEInTheMixedPspinModels}
Wei-Kuo Chen and Dmitry Panchenko.
\newblock On the tap free energy in the mixed p-spin models.
\newblock {\em Communications in Mathematical Physics}, 362(1):219--252, Aug
  2018.

\bibitem{CrisantiSommersTAPApproachtoSphericalPspinSGModels}
A~Crisanti and H-J Sommers.
\newblock {{T}houless-{A}nderson-{P}almer approach to the spherical p-spin spin
  glass model}.
\newblock {\em Journal de Physique I}, 5(7):805--813, 1995.

\bibitem{CrisantiSommersTheSphericalPspinInteractionSGModel}
Andrea Crisanti and H-J Sommers.
\newblock {The spherical p-spin interaction spin glass model: the statics}.
\newblock {\em Zeitschrift f{\"u}r Physik B Condensed Matter}, 87(3):341--354,
  1992.

\bibitem{GeneoveseTantariLegendrDualityOfSphericalandGaussianSpinGlasses}
Giuseppe Genovese and Daniele Tantari.
\newblock {Legendre duality of spherical and Gaussian spin glasses}.
\newblock {\em Mathematical Physics, Analysis and Geometry}, 18(1):10, 2015.

\bibitem{GuerraBrokenReplicaSymmetryBounds}
Francesco Guerra.
\newblock {Broken replica symmetry bounds in the mean field spin glass model}.
\newblock {\em Comm. Math. Phys.}, 233(1):1--12, 2003.

\bibitem{KosterlitzThoulessJonesSphericalModelofASpinGlass}
JM~Kosterlitz, DJ~Thouless, and Raymund~C Jones.
\newblock {Spherical model of a spin-glass}.
\newblock {\em Physical Review Letters}, 36(20):1217, 1976.

\bibitem{MezardParisiVirasoro-SpinGlassTheoryandBeyond}
Marc M{\'e}zard, Giorgio Parisi, and Miguel~Angel Virasoro.
\newblock {\em {Spin glass theory and beyond}}, volume~9 of {\em {World
  Scientific Lecture Notes in Physics}}.
\newblock World Scientific Publishing Co., Inc., Teaneck, NJ, 1987.

\bibitem{PanchenkoTheParisiUltrametricityConjecture}
Dmitry Panchenko.
\newblock {The {P}arisi ultrametricity conjecture}.
\newblock {\em Ann. of Math. (2)}, 177(1):383--393, 2013.

\bibitem{PanchenkoTheSKModel}
Dmitry Panchenko.
\newblock {\em {The Sherrington-Kirkpatrick model}}.
\newblock Springer Science \& Business Media, 2013.

\bibitem{PlefkaALBfortheSpinGlassOrderParameter}
T~Plefka.
\newblock {A lower bound for the spin glass order parameter of the
  infinite-ranged Ising spin glass model}.
\newblock {\em Journal of Physics A: Mathematical and General}, 15(5):L251,
  1982.

\bibitem{PlefkaConvergenceCondOftheTAPequations}
T~Plefka.
\newblock {Convergence condition of the {TAP} equation for the infinite-ranged
  {I}sing spin glass model}.
\newblock {\em Journal of Physics A: Mathematical and general}, 15(6):1971,
  1982.

\bibitem{SKSolvableModelOfASpinGlass}
David Sherrington and Scott Kirkpatrick.
\newblock {Solvable model of a spin-glass}.
\newblock {\em Physical review letters}, 35(26):1792, 1975.

\bibitem{SubagGeometryOfGibbsMeasure}
Eliran Subag.
\newblock {The geometry of the {G}ibbs measure of pure spherical spin glasses}.
\newblock {\em Invent. Math.}, 210(1):135--209, 2017.

\bibitem{talagrand2003spin}
M.~Talagrand.
\newblock {\em {Spin glasses: a challenge for mathematicians: cavity and mean
  field models}}, volume~46.
\newblock Springer, 2003.

\bibitem{TalagrandFEOftheSphericalMeanFieldModel}
Michel Talagrand.
\newblock {Free energy of the spherical mean field model}.
\newblock {\em Probab. Theory Related Fields}, 134(3):339--382, 2006.

\bibitem{TalagrandTheParisiFormula}
Michel Talagrand.
\newblock {The {P}arisi formula}.
\newblock {\em Ann. of Math. (2)}, 163(1):221--263, 2006.

\bibitem{TaoTopicsInRMT}
Terence Tao.
\newblock {\em {Topics in random matrix theory}}, volume 132.
\newblock American Mathematical Soc., 2012.

\bibitem{TAPSolutionOfSolvableModelOfASpinGlass}
David~J Thouless, Philip~W Anderson, and Robert~G Palmer.
\newblock {Solution of 'solvable model of a spin glass'}.
\newblock {\em Philosophical Magazine}, 35(3):593--601, 1977.

\bibitem{VilfanLectureNotesinStatisticalMechanics}
Igor Vilfan.
\newblock {Lecture Notes in Statistical Mechanics}.
\newblock \url{http://www-f1.ijs.si/~vilfan/SM/}.

\end{thebibliography}

\end{document}